\documentclass[11pt]{article}

\usepackage{amssymb, amsmath}

\newcommand{\R}{\mathbb R}
\newcommand{\C}{\mathbb{C}}
\newcommand{\Z}{\mathbb{Z}}

\newcommand{\Q}{\mathbb{Q}}
\renewcommand{\P}{\mathbb{P}}

\newcommand{\cA}{\mathcal{A}}

\newtheorem{prop}{Proposition}
\newtheorem{thm}[prop]{Theorem}
\newtheorem{rem}[prop]{Remark}

\newtheorem{lem}[prop]{Lemma}
\newtheorem{exmp}[prop]{Example}

\newenvironment{proof}{
\trivlist \item[\hskip \labelsep\mbox{\it Proof:
}]}{\hfill\mbox{$\square$}
\endtrivlist}

\title{Affine solution sets of sparse polynomial systems
\footnote{Partially supported by the Argentinian research grants
CONICET PIP 099/11 and UBACYT 20020090100069 (2010-2013).}}
\author{Mar\'\i a Isabel Herrero$^{\sharp}$, Gabriela Jeronimo$^{\sharp,\diamond}$, Juan
Sabia$^{\dag,\diamond}$\\[5mm]
{\small $\sharp$ Departamento de Matem\'atica, Facultad de Ciencias
Exactas y
Naturales,} \\[-1mm] {\small Universidad de Buenos Aires, Ciudad
Universitaria, (1428) Buenos Aires, Argentina}\\[2mm]
{\small $\dag$ Departamento de Ciencias Exactas, Ciclo B\'asico
Com\'un,}\\[-1mm]
{\small Universidad de Buenos Aires, Ciudad Universitaria, (1428)
Buenos Aires, Argentina}\\[2mm]
{\small $\diamond$ IMAS, CONICET, Argentina }}

\date{}

\begin{document}

\maketitle

\begin{abstract}
This paper focuses on the equidimensional decomposition of affine
varieties defined by sparse polynomial systems. For generic systems
with fixed supports, we give combinatorial conditions for the
existence of positive dimensional components which characterize the
equidimensional decomposition of the associated affine variety. This
result is applied to design an equidimensional decomposition
algorithm for generic sparse systems. For arbitrary sparse systems
of $n$ polynomials in $n$ variables with fixed supports, we obtain
an upper bound for the degree of the affine variety defined and we
present an algorithm which computes finite sets of points
representing its equidimensional components.
\end{abstract}

\bigskip


\section{Introduction}

The aim of this paper is to describe the affine solution set of a
polynomial system taking into account the sets of exponents of the
monomials with nonzero coefficients in the polynomials involved,
that is, their support sets.

Bernstein (\cite{Ber75}),  Kushnirenko (\cite{Kus76}) and Khovanskii (\cite{Kho78}) proved that the
number of isolated solutions in $(\C^* )^n$ of a polynomial system
with $n$ equations in $n$ variables is bounded by a combinatorial
invariant (the mixed volume) associated with their supports. This
result, which may be considered  the basis for the current study of
sparse polynomial systems, hints at the fact that the algorithms
solving these systems should have shorter computing time than the
general ones.

There are several algorithms to compute either numerically or
symbolically the isolated roots of sparse polynomial systems in
$(\C^* )^n$ (see, for example,\cite{VVC94, HS95, Roj03,
JMSW09}). The efficiency of some of these algorithms relies on the
use of polyhedral deformations  preserving the monomial structure of
the polynomial system under consideration.

The first step towards the study of the solutions of sparse systems
in the \emph{affine} case was to obtain upper bounds for the number
of isolated solutions in $\C^n$ in terms of the structure of their
supports and to design numerical algorithms to compute them
(see \cite{Roj94, LW96, RW96, HS97, EV99, GLW99}). Symbolic
algorithms performing this task were given in \cite{JMSW09} and
\cite{HJS10}.

The next natural step is to characterize the components of higher
dimension of the affine variety defined by a system of sparse
polynomial equations taking into account their supports. In this
context, in \cite{Ver09}, certificates for the existence of curves
are given in the numerical framework.

There are different symbolic algorithms describing the
equidimensional decomposition of a variety  which only take into
consideration the degrees of the polynomials defining it and not
their particular monomial structure. The earliest
deterministic ones can be found in \cite{ChG83} and \cite{GH91}
(see also \cite{GTZ88}, where the more general problem of the primary
decomposition of ideals is considered).
Probabilistic algorithms with shorter running time are given in
\cite{EM99} and \cite{JS02}. The complexities of these
probabilistic algorithms are polynomial in the B\'ezout number of
the system, which, in the generic case, coincides with the degree of
the variety the system defines. Other probabilistic algorithms are
presented in \cite{Lec03} and \cite{JKSS04} with complexities
depending on a new invariant related to the system (the
\emph{geometric degree}) which refines the B\'ezout bound.
Some of these algorithms can be derandomized easily via the Schwartz-Zippel
lemma (\cite{Schwartz80, Zippel93}) provided upper bounds for the degrees of
the polynomials characterizing exceptional instances are known.

Algorithms dealing with the problem from the numerical point of view can be traced back to
\cite{SW96}. A series of papers by Sommese, Verschelde and Wampler present successive improvements
to this procedure, leading to the irreducible decomposition algorithm
based on homotopy continuation described in \cite{SW05} (see references therein).

In this paper we analyze, both from the theoretic and algorithmic
points of view, the equidimensional decomposition of the affine
variety defined by a sparse polynomial system.

First, we consider the case of generic sparse systems. In this
context, there exists a major difference with the  case of dense
polynomials. The set of solutions of a generic system of $n$
polynomials in $n$ variables with fixed degrees consists only of
isolated points. However, fixing the set of supports of the $n$
polynomials in $n$ variables involved in a sparse system, for
generic choices of its coefficients, there may appear affine
components of positive dimension (see, for instance, Examples
\ref{ejemplo1} and \ref{ex:genericdecomp}  below). We show that the
existence of these generic components of positive dimension depends
only on the combinatorial structure of the supports: in Proposition
\ref{prop:components} below, we give conditions that yield these
components. Such conditions provide not only a theoretic description
of the equidimensional decomposition of the affine variety
$V(\mathbf{f})$ defined by a \emph{generic} sparse system
$\mathbf{f}$ in terms of the solution sets in the torus of smaller
systems $\mathbf{f}_I$ associated to subsets $I\subset \{1,\dots,
n\}$ but also a formula for the degree of $V(\mathbf{f})$ in this
generic case (see Theorem \ref{teoremacomb} below). Previous results
on this subject can be found in \cite{CD07}. There,  using also  a
combinatorial approach, the authors analyze thoroughly the problem
of deciding whether a system of $n$ binomials in $n$ variables has a
finite number of affine solutions and, in this case, the
computational complexity of the corresponding counting problem.

Our result is used to design a probabilistic algorithm which, for a
generic sparse system $\mathbf{f}$, computes the equidimensional
decomposition of $V(\mathbf{f})$ with a complexity depending on its
degree  and combinatorial invariants associated with the system (see
Theorem \ref{thm:gendecomp}). The idea of the algorithm is to
compute first a family of subsets $I \subset \{1,\dots, n\}$ which
may lead to components of $V(\mathbf{f})$ and solve the
corresponding polynomial systems $\mathbf{f}_I$ by applying symbolic
polyhedral deformations (\cite{JMSW09}) and a Newton-Hensel based
procedure (\cite{GLS01}). The output of the algorithm is, for each
$k=0,\dots, n-1$, a list of \emph{geometric resolutions}
representing the equidimensional component of dimension $k$ of
$V(\mathbf{f})$. A geometric resolution of an equidimensional
variety is a parametric type description of the variety which is
widely used in symbolic computations (see, for
instance, \cite{GHM+98, Rou99, GLS01, Sch03}); in Section
\ref{sec:definitions} below we give the precise definition we use.

The next step is to consider the equidimensional decomposition of
the affine variety defined by an \emph{arbitrary} system of sparse
polynomials. A question to answer beforehand is which parameter
should be involved in the algebraic complexity of an algorithm
solving this task. From previous experience, a natural invariant
expected to appear in the complexity bounds is the degree of the
variety, which is, in particular,  an upper bound for the number of
its irreducible components.

Unlike the B\'ezout bound for dense polynomials, in the sparse
setting, the degree of the affine variety defined by a generic
square system is \emph{not} an upper bound for the degree of the
variety defined by any system with the same supports (see Example
\ref{ejemplo2}). In \cite{KPS01}, a bound for the degree of the
affine variety defined by an arbitrary sparse polynomial system
depending on a mixed volume related to the \emph{union} of the
supports of the polynomials is presented (see also
\cite[Theorem 1]{Roj03} for a related result). Here, we obtain a sharper bound for
this degree also given by  a mixed volume associated to the supports
but not involving their union:

\bigskip
\noindent \textbf{Theorem.} \emph{Let $\mathbf{f} = (f_1, \dots,
f_n)$ be $n$ polynomials in $\C[X_1,\dots, X_n]$ supported on $\cA=
(\cA_1,\dots, \cA_n)$ and let $V(\mathbf{f}) = \{ x \in \C^n \mid
f_i( x) = 0 \hbox{ \ for every  } 1 \le i \le n \}.$  Then
$$ \deg (V(\mathbf{f})) \le MV_n(\cA_1 \cup \Delta,\dots, \cA_n\cup \Delta),$$
where $\Delta = \{ 0, e_1, \dots, e_n\}$ with $e_i$ the $i$th vector
of the canonical basis of $\R^n$ and $MV_n$ stands for the
$n$-dimensional mixed volume.}

\bigskip

Finally, we obtain an algorithm which, using a polyhedral
deformation, describes points in every irreducible component of the
affine variety defined by an arbitrary square sparse system with
complexity depending on the degree bound previously stated.

The idea of the algorithm relies on the fact that cutting the
variety with a generic affine linear variety of codimension $k$,
sufficiently many points in each irreducible component of dimension
$k$ can be obtained. To keep the complexity  within the desired
bounds, instead of computing this intersection, we proceed in a
particular way which enables us to compute a finite superset of the
intersection included in the variety. Representing a positive
dimensional variety by means of a finite set of points is a
well-known approach in numerical algebraic geometry (see the
notion of \emph{witness point supersets} in \cite{SW05}).

\bigskip

\noindent \textbf{Theorem.} \emph{Let $\mathbf{f} = (f_1, \dots,
f_n)$ be $n$ polynomials in $\Q[X_1,\dots, X_n]$ supported on $\cA=
(\cA_1,\dots, \cA_n)$. There is a probabilistic algorithm which,
taking as input the sparse representation of $\,\mathbf{f}$ computes
a family of $n$ geometric resolutions $(R^{(0)}, R^{(1)}, \dots,
R^{(n-1)})$ such that, for every $0\le k \le n-1$,  $R^{(k)}$
represents a finite set  containing $\deg V_k(\mathbf{f})$ points in
the equidimensional component $V_k(\mathbf{f})$ of dimension $k$ of
$V(\mathbf{f})$. The number of arithmetic operations over $\Q$
performed by the algorithm is of order $O\,\widetilde{\ }( n^4 d N
D^2)$, where $d= \max_{1\le j \le n}\{ \deg(f_j)\}$, $N=
\sum_{j=1}^n \#(\cA_j \cup \Delta)$ and $D=MV_n(\cA_1 \cup
\Delta,\dots, \cA_n\cup \Delta)$.}

\bigskip

Here $O\,\widetilde{\ }$ refers to the standard soft-oh notation
which does not take into account logarithmic factors. Furthermore,
we have ignored factors depending polynomially on the size of
certain combinatorial objects associated to the polyhedral
deformation. For a precise complexity statement, see Theorem
\ref{thm:nongeneric}, and for error probability considerations, see
Remark \ref{rem:prob}.

\bigskip

The paper is organized as follows. In Section \ref{sec:prelim}, the
basic definitions and notations used throughout the paper are
introduced. Section \ref{sec:generic} is devoted to the
equidimensional decomposition of affine varieties defined by generic
sparse systems: first, we consider the solution  sets in $(\C^*)^n$
of underdetermined systems (see Subsection \ref{sec:toriccomps});
then, we prove our main theoretic result on equidimensional
decomposition and present our algorithm to compute it (see
Subsection \ref{sec:affinecomps}). Finally, in Section
\ref{sec:arbitrary}, we consider the case of arbitrary sparse
systems: we prove the upper bound for the degree of affine varieties
defined by these systems in Subsection \ref{sec:degreebound} and, in
Subsection \ref{sec:algnongeneric}, we describe our algorithm to
compute representative points of the equidimensional components.

\section{Preliminaries}\label{sec:prelim}

\subsection {Basic definitions and notation}\label{sec:definitions}

Throughout this paper, unless otherwise explicitly stated, we deal with polynomials in $\Q[X_1, \dots,
X_n]$, that is to say polynomials with rational coefficients in $n$
variables $X=(X_1,\dots, X_n)$. If $\mathbf{f}=(f_1, \dots, f_s)$ is
a family of such polynomials, $V(\mathbf{f})$ will denote the
algebraic variety of their common zeroes in $\C^n$, the
$n$-dimensional affine space over the complex numbers.

The algebraic variety $V(\mathbf{f})\subset \C^n$ can be decomposed
uniquely as a finite union of irreducible varieties in a
non-redundant way. This leads to the \emph{equidimensional
decomposition} of the variety:
$$V(\mathbf{f}) = \bigcup_{k=0}^n V_k(\mathbf{f}),$$ where,
for every $0\le k \le n$, $V_k(\mathbf{f})$ is the (possibly empty)
union of all the irreducible components of dimension $k$ of
$V(\mathbf{f})$.

 The degree of each equidimensional component $V_k(\mathbf{f})$ is the number of points in its intersection with
 a generic affine linear variety of codimension $k$, and the degree of $V(\mathbf{f})$,
 which we denote by $\deg(V(\mathbf{f}))$, is the
 sum of the degrees of its equidimensional components (see \cite{Hei83}).

A common way to describe zero-dimensional affine varieties defined
by polynomials over $\Q$ is a \emph{geometric resolution}
(see, for instance, \cite{GLS01} and the references therein). The
precise definition we are going to use is the following:

Let $V= \{\xi^{(1)}, \dots, \xi^{(D)}\}\subset \C^n$ be a
zero-dimensional variety defined by rational polynomials. Given a
linear form $\ell = \ell_1X_1 + \dots + \ell_n X_n$ in $\Q[X_1,
\dots, X_n]$ such that $\ell(\xi^{(i)})\ne \ell(\xi^{(j)})$ if $i
\ne j$, the following polynomials completely characterize $V$:
\begin{itemize}
\item the minimal polynomial $q = \prod_{1 \le i \le D}(u - \ell(\xi^{(i)})) \in \Q[u]$ of $\ell$ over
the variety $V$ (where $u$ is a new variable),
\item polynomials $v_1, \dots, v_n \in \Q[u]$ with $\deg(v_j) < D$ for every $1 \le j \le n$
satisfying $V = \{(v_1 (\eta), \dots, v_n (\eta)) \in \C^n \mid \eta
\in \C, q(\eta)=0\}$.
\end{itemize}
The family of univariate polynomials $(q, v_1,\dots, v_n) \in
\Q[u]^{n+1}$ is called a geometric resolution of $V$ (associated
with the linear form $\ell$).

An equivalent description of $V$ can be given through the so-called
\emph{Kronecker representation} (see \cite{GLS01}), which
consists of the minimal polynomial $q$ and polynomials $w_1, \dots,
w_n \in \Q[u]$ such that $V = \{(\frac{w_1}{q'} (\eta), \dots,
\frac{w_n}{q'} (\eta)) \in \C^n \mid \eta \in \C, q(\eta)=0\}$,
where $q'$ is the derivative of $q$. Either representation can be
obtained from the other one in polynomial time.

The notion of geometric resolution can be extended to any
equidimensional variety:

Let $V\subset \C^{n}$ be an equidimensional variety of dimension $r$
defined by polynomials $f_1,\dots, f_{n-r}\in\Q[X_1,\dots, X_n]$. Assume that for
each irreducible component $C$ of $V$, the identity $I(C) \cap
\Q[X_1,\dots, X_r] = \{ 0 \}$ holds, where $I(C)$ is the ideal of all polynomials in $\Q[X_1,\dots, X_n]$ vanishing identically over C. Let $\ell$ be a nonzero linear
form in $\Q[X_{r+1},\dots, X_n]$ and $\pi_\ell:V\to\C^{r+1}$ the
morphism defined by $\pi_\ell(x)=(x_1,\dots, x_r,\ell(x))$. Then,
there exists a unique (up to scaling by nonzero elements of $\Q$)
polynomial $Q_\ell\in\Q[X_1,\dots, X_r,u]$ of minimal degree
defining $\overline{\pi_\ell(V)}$. Let $q_\ell\in\Q(X_1,\dots,
X_r)[u]$ denote the (unique) monic multiple of $Q_\ell$ with $\deg_u
(q_\ell)=\deg_u (Q_\ell)$. In these terms, if $\ell$ is a
\emph{generic} linear form, a \emph{geometric resolution} of $V$ is
$(q_\ell, v_{r+1}, \dots, v_n) \in (\Q(X_1,\dots, X_r)[u])^{n-r+1}
$, where, for $r+1\le i\le n$, $v_{i} $ satisfies
$$\frac{\partial q_\ell}{\partial u}( \ell) \, X_{i}= v_{i} ( \ell) \ \mbox{ in  }
\Q(X_1,\dots, X_r)\otimes\Q[V]$$ and $\deg_u (v_{i})< \deg_u (q_\ell
)$.

\subsection{Algorithms and codification}

Although we work with polynomials, our algorithms only deal with
elements in $\Q$. The notion of \emph{complexity} of an algorithm we
consider is the number of operations and comparisons in $\Q$ it has
to perform. We will encode multivariate polynomials in different
ways:
\begin{itemize}
\item in sparse form, that is, by means of the list of pairs $(a, c_a)$ where $a$ runs
over the set of exponents of the monomials appearing in the
polynomial with nonzero coefficients and $c_a$ is the corresponding
coefficient,
\item in the standard dense form, which encodes a polynomial
as the vector of its coefficients including zeroes (we use this
encoding only for univariate polynomials),
\item  in the \emph{straight-line program} (slp for short) encoding. A straight-line
program is an algorithm without branchings which allows the
evaluation of the polynomial at a generic value (for a precise
definition and properties of slp's, see \cite{BCS97}).
\end{itemize}

In our complexity estimates, we will use the usual $O$ notation: for
$f,g\colon \Z_{\ge 0} \to \R$,  $f(d) = O(g(d))$ if $|f (d)| \le c|g(d)|$
for a positive constant $c$. We will also use the notation $M(d) = d
\log^2(d) \log(\log(d))$, where $\log$ denotes logarithm to base
$2$. We recall that multipoint evaluation and interpolation of
univariate polynomials of degree $d$  with coefficients in a
characteristic-0 commutative ring $R$  can be performed with
$O(M(d))$ operations and that multiplication and division with
remainder of such polynomials can be done with $O(M(d)/ \log(d))$
arithmetic operations in $R$.

We denote by $\Omega$ the exponent in the complexity estimate
$O(d^\Omega)$ for the multiplication of two $d\times d$ matrices
with rational coefficients. It is known that $\Omega < 2.376$
(see \cite[Chapter 12]{Gat99}). Finally, we write $\overline \Omega$
for the exponent ($\overline \Omega <4$) in the complexity
$O(d^{\overline \Omega})$ of  operations on $d\times d$ matrices
with entries in a commutative ring $R$.

Our algorithms are probabilistic in the sense that they make random
choices of points  which lead to a correct computation provided the
points lie outside  certain proper Zariski closed sets of  suitable
affine spaces. Then, using the Scwhartz-Zippel lemma (\cite{Schwartz80, Zippel93}),
the error probability of our algorithms
can be controlled by making these random choices within sufficiently
large sets of integer numbers whose size depend on the degrees of the polynomials
defining the previously mentioned Zariski closed sets.

\subsection {Sparse systems}

Given a family  $\cA = (\cA_1,\ldots,\cA_s)$ of  finite subsets of
$(\Z_{\ge 0})^n$, a \emph{sparse polynomial system  supported on}
$\cA$  is given by polynomials $ f_j =\sum_{a\in\cA_j}c_{j,a}\,
X^a$ in the variables $X=(X_1,\ldots,X_n)$, with $c_{j,a}\in
\C\setminus\{0\}$ for each $a\in\cA_j$ and $1\le j\le s$. We write
$\mathbf{f}=(f_1,\dots, f_s)$ for this system.

Assume $s=n$. We denote by $MV_n(\cA)$ the \emph{mixed volume} of
the convex hulls of $\cA_1,\ldots,\cA_n$ in $\R^{n}$ (see, for
example, \cite[Chapter 7]{CLO} for the definition), which is an upper
bound for the number of isolated roots in $(\C^*)^n$ of a sparse
system supported on $\cA$  (see \cite{Ber75}).

The mixed volume $MV_n(\mathcal{A})$ can be computed as the sum of
the $n$-dimensional volumes of the convex hulls of all the
\emph{mixed cells} in a fine mixed subdivision of $\cA$. Such a
subdivision can be obtained by means of a standard lifting process
(see \cite[Section 2]{HS95}): let $\omega= (\omega_1,\dots,
\omega_n)$ be a $n$-tuple of generic functions $\omega_j: \cA_j\to
\R$ and consider the polytope $P$ in $\R^{n+1}$ obtained by taking
the pointwise sum of the convex hulls of the graphs of $\omega_j$
for $1\le j\le n$. Then, the projection of the lower facets of $P$
(that is, the $n$-dimensional faces with inner normal vector with a
positive last coordinate) induces a fine mixed subdivision of $\cA$.
The dynamic enumeration procedure described in \cite{MTK07} appears
to be the fastest algorithm known up until now to achieve this
computation of mixed cells.

The \emph{stable mixed volume} of $\cA$, which is denoted by
$SM_n(\cA)$, is introduced in \cite{HS97} as an upper bound for the
number of isolated roots in $\C^n$ of a sparse system supported on
$\cA$. Consider $\cA^{0} = (\cA_1^{0}, \dots, \cA_n^{0})$ the family with
$\cA_j^{0} := \cA_j \cup \{0\}$ for every $1\le j \le n$, and let  $\omega^{0} =
(\omega^{0}_1, \dots, \omega^{0}_n)$ be a lifting for $\cA^{0}$
defined by $\omega^{0}_j(q) = 0$ if $q \in \cA_j$ and
$\omega^{0}_j(0)= 1$ if $0 \notin \cA_j$. The stable mixed volume of
$\cA$ is defined as the sum of the mixed volumes of all the cells in
the subdivision of $\cA^{0}$ induced by $\omega^{0}$ corresponding
to facets having inner normal vectors with non-negative entries.

\section{Generic sparse systems}\label{sec:generic}

\subsection{Toric components}\label{sec:toriccomps}

Let $n$ and $m$ be positive integers and let $\mathcal{A} = (\cA_1,
\dots, \cA_m)$ be a family of finite subsets of $(\Z_{\ge 0})^n$.
For $1 \le j \le m$, let
$$ F_j (C_j, X) = \sum_{a \in \cA_j} C_{j,a}X^a $$
where  $X =(X_1, \dots, X_n)$; for $a = (a_1, \dots, a_n)$, $X^a=
\prod_{1 \le j \le n}X_j^{a_j}$, and $C_j = (C_{j,a})_{a \in \cA_j}$
are $N_j = \# \cA_j$ indeterminate coefficients.

Following \cite{Stu94}, consider the incidence variety
$$ \{(x,c) \in (\C^*)^n \times (\P^{N_1-1}\times \dots \times \P^{N_m-1}) \mid F_j(c_j, x) = 0  \hbox{ \ for every \ } 1 \le j \le m  \}$$
and its projection to the second factor
$$Z = \{c \in  \P^{N_1-1}\times \dots \times \P^{N_m-1} \mid \exists \ x \in (\C^*)^n  \hbox{\ with \ }  F_j(c_j, x) = 0  \hbox{  for every  } 1 \le j \le m  \}.$$

Note that the elements in $Z$ correspond essentially to coefficients of
systems supported on $\cA$ which have a solution in $(\C^*)^n$.

\begin{lem}\label{genericdimension}
The Zariski closure of $Z$ equals $\P^{N_1-1}\times \dots \times
\P^{N_m-1}$ if and only if, for every  $J \subseteq \{ 1,\dots,
m\}$, $\dim \left(\sum_{j \in J} \cA_j\right) \ge \# J$. In
particular, if $m > n$,  a generic system supported on $\cA$ has no
solutions in $(\C^*)^n$. Moreover, if $m \le n$ and  $\dim
\left(\sum_{j \in J} \cA_j\right) \ge \# J$ for every $J \subseteq
\{ 1,\dots, m\}$, the solution set in $(\C^*)^n$ of a generic system
supported on $\cA$ is an equidimensional variety of dimension $n-m$
and degree $MV_n(\cA_1,\dots, \cA_m, \Delta^{(n-m)})$, where $\Delta
= \{ 0,e_1,\dots, e_n\}$  with $e_i$ the $i$th vector of the
canonical basis of $\R^n$ and the superscript $(n-m)$ indicates that
it is repeated $n-m$ times.
\end{lem}

\begin{proof} The first statement of the Lemma follows as in
\cite[Theorem 1.1]{Stu94}.

Assume that $m\le n$ and for every  $J \subseteq \{ 1,\dots, m\}$,
$\dim \left(\sum_{j \in J} \cA_j\right) \ge \# J$. Then, if
$\widetilde \cA = (\cA_1,\dots, \cA_m, \Delta^{(n-m)})$, we have
that for every $\widetilde{J} \subseteq \{ 1,\dots, n\}$, the
inequality $\dim \left(\sum_{j \in \widetilde{J}} \widetilde
\cA_j\right) \ge \# \widetilde{J}$ holds, since $\dim (\Delta) = n$.
Therefore, $MV_n (\widetilde \cA) >0$, which implies that a generic
system supported on $\widetilde \cA$ has finitely many solutions in
$(\C^*)^n$ (as many as $MV_n (\widetilde \cA)$).

Now, the solution set of a generic system supported on $\widetilde
\cA$ is the intersection of the solution set of a generic system of
$m$ equations supported on $\cA$, which is a variety of dimension at
least $n-m$ in $(\C^*)^n$, and $n-m$ generic hyperplanes. We
conclude that the solution set in $(\C^*)^n$ of a generic system
supported on $\cA$ is an equidimensional variety of dimension $n-m$
and degree $MV_n(\widetilde \cA)$.
\end{proof}

Assume now that $\mathbf{f} = (f_1, \dots, f_m)$ are generic
polynomials in the variables $X=(X_1,\dots, X_n)$ supported on $\cA=
(\cA_1,\dots, \cA_m)\subset (\Z^n_{\ge 0})^m$, with $m\le n$.  The
previous lemma states that the affine variety
$V^*(\mathbf{f})\subset \C^n$ consisting of the union of all the
irreducible components of $V(\mathbf{f}) =\{ x \in \C^n \mid f_j( x)
= 0 \hbox{ for every  } 1 \le j \le s \} $ that have a non-empty
intersection with $(\C^*)^n$ is either the empty set or an
equidimensional variety of dimension $n-m$.

In what follows, we extend the symbolic algorithm from
\cite[Section 5]{JMSW09}, which deals with the case $m=n$, to a
procedure for the computation of a geometric resolution of
$V^*(\mathbf{f})$ for arbitrary $m\le n$. As in \cite{JMSW09}, our
algorithm assumes that a fine mixed subdivision of $(\cA,
\Delta^{(n-m)})$ induced by a generic lifting function $\omega=
(\omega_1,\dots,\omega_n)$ is given by a pre-processing.

\bigskip
\noindent\hrulefill

\noindent \textbf{Algorithm} \texttt{GenericToricSolve}

\medskip
\noindent  INPUT: A sparse representation of a generic system
$\mathbf{f} =(f_1,\dots, f_m)$ in the variables $X=(X_1,\dots, X_n)$
supported on $\cA= (\cA_1,\dots, \cA_m)$,  a lifting function
$\omega =(\omega_1, \dots, \omega_n)$ and the mixed cells in the
subdivision of $(\cA, \Delta^{(n-m)})$ induced by $\omega$.

\smallskip
\begin{enumerate}
\item\label{alg:empty} If the fine mixed subdivision of $(\cA, \Delta^{(n-m)})$ does not contain any mixed cell, return $R=\emptyset$. Otherwise, continue to Step \ref{alg:matrixA}.
\item\label{alg:matrixA} Choose randomly the entries of a matrix $A = (a_{hl}) \in \Q^{n \times n}$ and  a vector $b = (b_1,\dots, b_{n}) \in \Q^{n}$.
\item\label{alg:LinearForms} For $1 \le h \le n-m$, consider the affine linear
forms $L_h = \sum_{l=1}^n a_{hl} X_l - b_h$.
\item\label{alg:ToricSolve0} Apply  \cite[Algorithm 5.1]{JMSW09} to obtain a geometric resolution
$(q(u),v_1(u), \dots, v_{n}(u))$ of the isolated common zeroes of
$\mathbf{f}, L_1, \dots, L_{n-m}$ in $(\C^*)^{n}$.
\item \label{alg:slp_g} Obtain an slp for the polynomials $\mathbf{g}:=\mathbf{f}(A^{-1}Y)$ in the new variables $Y = (Y_1, \dots, Y_{n})$.
\item \label{alg:ChangeGR} Compute $(w_1(u), \dots, w_{n}(u))^t:=A(v_1(u), \dots,
v_{n}(u))^t$.
\item \label{alg:NewtonGLS}  Apply the Global Newton Iterator from
\cite[Algorithm 1]{GLS01} to the polynomials $\mathbf{g}(Y)$, the
geometric resolution $(q(u), w_{n-m+1}(u), \dots, w_{n}(u))$ of
$V(\mathbf{g}(b, Y_{n-m+1},\dots, Y_n))$,
 and precision $ \kappa=
MV_{n}(\cA, \Delta^{(n-m)})$ to obtain a geometric resolution $R_Y$
of an equidimensional variety of dimension $n-m$.
\item \label{alg:ChangeGR2} Obtain the geometric resolution
$R:=A^{-1}{R_Y}$ of $V^*(\mathbf{f})$.
\end{enumerate}

\smallskip
\noindent OUTPUT: The geometric resolution $R$ of $V^*(\mathbf{f})$.

\noindent\hrulefill

\bigskip

In the sequel we will justify the correctness of the above procedure
and estimate its complexity.

Since $\mathbf{f}$ is a  generic sparse system of $m$ equations in
$n$ variables, as a consequence of Lemma \ref{genericdimension}, if
$L_1,\dots, L_{n-m}$ are generic linear forms,  $V^*(\mathbf{f})\cap
V(L_1, \dots, L_{n-m})$ is either the empty set (when
$V^*(\mathbf{f})$ is the empty set) or a finite set consisting of
$\deg (V^*(\mathbf{f}))$ points (when $V^*(\mathbf{f})$ is not the
empty set). Step \ref{alg:empty} decides whether
$V^*(\mathbf{f})\cap V(L_1, \dots, L_{n-m})$ is empty or not, since
the mixed volume of the supports of these polynomials is the sum of
the volumes of the mixed cells in a fine mixed subdivision
(\cite{HS95}). If it is not empty, this finite set of points can be
regarded as  a generic fiber of a generic linear surjective
projection and therefore, it enables us to recover the variety
$V^*(\mathbf{f})$ by deformation techniques.

Thus, the idea of the algorithm is to choose $n-m$ linear forms at
random, then compute a geometric resolution of the set
$V^*(\mathbf{f})\cap V(L_1, \dots, L_{n-m})$ and finally, apply a
Newton-Hensel lifting to the finite set obtained in order to get a
geometric resolution of $V^*(\mathbf{f})$.

Step \ref{alg:matrixA} deals with the random choice of the entries
of a matrix and a vector. This
random choice does not affect the overall complexity of the procedure (see
Remark \ref{rem:prob} below). In Step
\ref{alg:LinearForms}, the sparse encoding of $n-m$ linear forms
constructed from the previous data is obtained.

The idea of Step \ref{alg:ToricSolve0} is to obtain a geometric
resolution of $V^*(\mathbf{f})\cap V(L_1,\dots, L_{n-m})$. In order
to do this, the algorithm computes the isolated common zeroes in
$(\C^*)^n$ of the \emph{generic} system $\mathbf{f}, L_1,\dots,
L_{n-m}$ supported on $(\cA, \Delta^{(n-m)})$. Note that, if
$L_1,\dots, L_{n-m}$ are generic, this set of points meets only the
irreducible components of $V^*(\mathbf{f})$, that is, it contains no
point in the irreducible components of $V(\mathbf{f})$ with
vanishing coordinates. By applying the result in \cite[Proposition
5.13]{JMSW09}, it follows  that the complexity of this step is
$O\left((n^3 (N + (n-m)n) \log d +
n^{1+\Omega})M(D)M(\mathfrak{M})(M(D)+M(E))\right)$, where
\begin{itemize}\itemsep=0pt
\item $N:=\sum\limits_{1\le j \le m}\#\cA_j;$
\item $d:=\max_{1\le j \le m}\{ \deg(f_j)\};$
\item $D:=MV_{n}(\cA, \Delta^{(n-m)})$;
\item $\mathfrak{M}:=\max\{\|\mu\| \}$, where the maximum ranges over all primitive normal vectors to the mixed cells in the fine mixed subdivision of $(\cA, \Delta^{(n-m)})$ given by $\omega$;
\item $E:=MV_{n+1}(\Delta \times \{0\},\cA_1(\omega_1), \dots, \cA_m(\omega_m), \Delta(\omega_{m+1}), \dots,
\Delta(\omega_{n}))$, where $\cA_j(\omega_j)$ $(1\le j \le m)$ and
$\Delta (\omega_l)$ $(m+1\le l\le n)$ are, respectively, the
supports of $\mathbf{f}, L_1, \dots, L_{n-m}$ lifted by $\omega$.
\end{itemize}

Now the algorithm lifts the geometric resolution of the
zero-dimensional subset of $V(\mathbf{f})$ obtained so far to a
geometric resolution of the union of the irreducible components of
this variety having a non-empty intersection with $(\C^*)^{n}$. In
order to do this, we consider the change of variables given by $Y =
A. X$ and make this change of variables in the polynomials
$\mathbf{f}$ and the geometric resolution already obtained (Steps
\ref{alg:slp_g} and \ref{alg:ChangeGR}). A possible way of making this change of variables
is by first computing
$A^{-1}$ with $O(n^\Omega)$ operations and using it to obtain an slp of
length $L = O( n^2 + n\log(d) N) $ for the polynomials in
$\mathbf{g}$ (note that the length of this slp depends only on the cost $O(n^2)$ of computing the product of
$A^{-1}$ times a vector, and not on the cost of inverting $A$). Taking into account that the degrees of the
polynomials $v_1,\dots, v_{n}$ are bounded by $D$, to write the
geometric resolution in the new variables, we perform $O(n^2 D)$
operations.

Note that $(w_1(u),\dots, w_{n-m}(u))= b$ and $(q(u),w_{n-m+1}(u),
\dots, w_{n}(u))$ is a geometric resolution of the isolated points
in $V(\mathbf{g}(b, Y_{n-m+1}, \dots, Y_{n}))$ corresponding to the
isolated points in $(\C^*)^n$ of $V(\mathbf{f},L_1,\dots, L_{n-m})$.
Now, the  geometric resolution of $V^*(\mathbf{f})$ with respect to
the linear projection given by $Y$ consists of polynomials in
$\Q[Y_1,\dots, Y_{n-m},u]$ having total degrees bounded by $D$.
Therefore,  it suffices to compute the representatives of these
polynomials in $(\Q[Y_1,\dots, Y_{n-m}]/\langle Y_1-b_1,\dots,
Y_{n-m}-b_{n-m}\rangle^{D+1})[u]$. To this end, in Step
\ref{alg:NewtonGLS} we apply successively the Global Newton Iterator
from \cite{GLS01} to the polynomials $\mathbf{g}$, starting with
the geometric resolution $(q(u), w_{n-m+1}(u), \dots, w_{n}(u))$
obtained in Step \ref{alg:ChangeGR}, which can be regarded as a
representative in $(\Q[Y_1,\dots, Y_{n-m}]/\langle Y_1-b_1,\dots,
Y_{n-m}-b_{n-m}\rangle)[u]$, up to the required precision
$D=MV_{n}(\cA, \Delta^{(n-m)})$. Using \cite[Lemma 2]{GLS01} and
encoding the elements of  $\Q[Y_1, \dots, Y_{n-m}]/\langle Y_1-b_1,
\dots, Y_{n-m}-b_{n-m}\rangle^k$ as $(k+1)$-tuples of slp's (one slp
for each homogeneous component), the complexity of Step
\ref{alg:NewtonGLS} is of order $O((m L + m^{\overline \Omega}) M(D)
D^2)$.

Finally, the algorithm changes variables back in order to obtain the
desired geometric resolution of $V^*(\mathbf{f})$, which adds $O(n^2
D)$ to the complexity.

Taking into account the previous complexity estimates, we have the
following result:

\begin{prop}\label{prop:toric}
Let $\mathbf{f}=(f_1,\dots, f_m)$ be a system of $m\le n$ generic
polynomials in \linebreak $\Q[X_1,\dots, X_n]$ supported on $\cA=(\cA_1,\dots,
\cA_m)$.   \texttt{GenericToricSolve} is a probabilistic algorithm
that computes a geometric resolution of the affine variety
$V^*(\mathbf{f})$ consisting of the union of all the irreducible
components of $V(\mathbf{f})$ that have a non-empty intersection with
$(\C^*)^n$. Using the previous notation, the complexity of this
algorithm is of order $$ O\left( n^3 (N+ (n-m)n) \log(d) M(D)
(M(\mathfrak{M}) ( M(D)+ M(E)) + D^2) \right).
$$

\end{prop}

\subsection{Affine components} \label{sec:affinecomps}

\subsubsection{Theoretic results}

This section is devoted to showing a combinatorial description of
the equidimensional decomposition of the affine variety defined by a
\emph{generic} sparse polynomial system. More precisely, we prove
combinatorial conditions on the supports of the polynomials that
determine the existence of irreducible components of the different
possible dimensions not intersecting $(\C^*)^n$ and give a
combinatorial characterization of the set of linear subspaces where
these components lie. This characterization enables us to give a
combinatorial formula for the degree of these varieties.

\medskip

The following example shows that generic square sparse systems may
define affine varieties containing positive dimensional components.
It also shows that neither the mixed volume nor the stable mixed
volume of the system are upper bounds for the degree of the  affine
variety defined:

\begin{exmp}
\label{ejemplo1}{\rm Consider a generic sparse system supported on $\cA
= ( \cA_1 , \cA_2 , \cA_3 )$ where $\cA_1 = \{(1,1,2), (1,1,1)\}$,
$\cA_2 = \{(2,0,1), (1,0,1)\}$ and $\cA_3 = \{(0,2,1), (0,1,1)\}$:
$$\begin{cases} a X_1X_2X_3^2 + b  X_1X_2X_3 = 0 \\
c X_1^2X_3 + d  X_1X_3 = 0 \\ e X_2^2X_3 + f  X_2X_3 = 0
\end{cases}$$
with $a,b,c,d,e,f$  nonzero complex numbers. The affine variety defined by the system has $5$ irreducible
components of degree $1$: $\{ x_3 = 0 \}$, $\{x_1 = 0, x_2=-
\frac{f}{e} \}$,$\{ x_1=- \frac{d}{c}, x_2 = 0 \}$,  $\{x_1 = 0,
x_2=0 \}$ and  $\{(-\frac{d}{c},-\frac{f}{e},-\frac{b}{a}) \}$.
However, we have that $MV_3(\cA_1, \cA_2, \cA_3) = 1$ and   $SM_3(\cA_1, \cA_2,
\cA_3)\le MV_3(\cA_1 \cup \{0\}, \cA_2\cup \{0\}, \cA_3\cup \{0\}) =
4$.}
\end{exmp}

Let $\mathbf{f} = (f_1, \dots, f_s)$ be generic polynomials in the
variables $X=(X_1,\dots, X_n)$ supported on $\cA= (\cA_1,\dots,
\cA_s)\subset (\Z_{\ge 0}^n)^s$.

For $I \subset \{1,\dots,n\}$, we define $$J_{I} = \{ j \in
\{1,\dots, s\} \mid f_j |_{\bigcap_{i \in I}\{ x_i = 0 \}}
\not\equiv 0 \},$$ that is,  the set of indices of the polynomials
in $\mathbf{f}$ that do not vanish identically under the
specialization $X_i=0$ for every $i\in I$, and
$$ \mathbf{f}_I = ((f_j)_I)_{j \in J_I}$$ where, for a polynomial $f \in \C[X_1,\dots,X_n]$,
$f_I$ denotes the polynomial in $\C[(X_i)_{i \notin I}]$ obtained
from $f$ by specializing $X_i= 0$ for every $i \in I$. Namely,
$\mathbf{f}_I$ is the set of polynomials obtained by specializing
the variables indexed by $I$ to 0 in the polynomials in $\mathbf{f}$
and discarding the ones that vanish identically. We denote by
$\cA_j^I$ the support of $(f_j)_I$, by $\pi_I: \C^n \to \C^{n-\#I}$
the projection $\pi_I(x_1,\dots, x_n) = (x_i)_{i \notin I}$ onto the
coordinates not in $I$ and by $\varphi_I: \C^{n-\#I}\to \C^n$ the
map that inserts zeroes in the coordinates indexed by $I$.

For an irreducible subvariety $W$ of $V(\mathbf{f})= \{ x \in \C^n
\mid f_j( x) = 0  \hbox{ \ for every } 1 \le j \le s \}$, let
$$I_W =\{i \in \{1,\dots, n\} \mid W \subset \{ x_i = 0 \} \}.$$

\begin{lem}\label{lemma:IW}
Under the previous assumptions, let $W$ be an irreducible component
of $V(\mathbf{f})$. Then $\dim W = n - \# I_W - \# J_{I_W}.$
Moreover $\pi_{I_W}(W)$ is an irreducible component of
$V(\mathbf{f}_{I_W}) \subset \C^{n-\#I_W} $ intersecting
$(\C^*)^{n-\#I_W}$.
\end{lem}

\begin{proof}
Without loss of generality, we may assume that $I_W=\{r+1,\dots,
n\}$ and $J_{I_W}= \{1,\dots, m\}$ for some $r> 0$ and $m\le n$.
Then, $\pi:=\pi_{I_W}: \C^n\to \C^{r}$ is the projection to the
first $r$ coordinates and $\mathbf{f}_{I_W} =(f_1(x_1,\dots,
x_r,\mathbf{0}),\dots, f_m(x_1,\dots, x_r,\mathbf{0}))$, where
$\mathbf{0}$ is the origin of $\C^{n-r}$.

Note that $W= \pi(W) \times \{ \mathbf{0}\}$ and $\pi(W) \subset
V(\mathbf{f}_{I_W})$. If $\pi(W) \subset C\subset V(\mathbf{f}_{I_W})$ for
an irreducible component $C$ of $V(\mathbf{f}_{I_W})$, it follows that
$W \subset C\times \{ \mathbf{0}\} \subset V(\mathbf{f})$, with
$C\times \{ \mathbf{0}\}$ an irreducible variety. Since $W$ is an
irreducible component of $V(\mathbf{f})$, the equality $W = C\times
\{ \mathbf{0}\}$ holds. This implies that $\pi(W) = C$ is an
irreducible component of $V(\mathbf{f}_{I_W})$.

In addition, by the definition of $I_W$, we have that $W\cap
\left(\bigcap_{i=1}^r \{x_i\ne 0\}\right) \ne \emptyset $:
otherwise, $W\subset \bigcup_{i=1}^r \{x_i=0\}$, which implies that
$W\subset \{x_i=0\}$ for some $1\le i\le r$ since $W$ is an
irreducible variety. Therefore, $\pi(W) \cap (\C^*)^r\ne \emptyset$.

By Lemma \ref{genericdimension}, we conclude that $\pi(W) \cap
(\C^*)^r$ has dimension $r-m$ and so, $\dim(W) =  \dim(\pi(W))=
n-\#I_W - \#J_{I_W}$.
\end{proof}

The previous lemma allows us to prove
that a result established for binomials in \cite[Theorem 2.6]{CD07} also holds for arbitrary polynomials.

\begin{prop}
With our previous notation, assuming that $s= n$ and $0\in V(\mathbf{f})$, we have that $V(\mathbf{f})$ consists only of isolated
points if and only if for every $I\subset \{1, \dots,n\}$, $\# I + \# J_{I} \ge n$.
\end{prop}

\begin{proof} Assume that there exists  $I\subset \{1, \dots,n\}$ such that $\# I + \# J_{I} < n$. Since $0 \in V(\mathbf{f}_I)\subset \C^{n-\#I}$ and this variety is defined by $\#J_I$ polynomials in $n-\# I$ variables, it follows that $\dim (V(\mathbf{f}_I)) \ge n-\#I- \#J_I>0$. Taking into account that $V(\mathbf{f}) \supseteq \varphi_I(V(\mathbf{f}_I))$, we conclude that $\dim(V(\mathbf{f}))>0$.

Conversely, if $\dim(V(\mathbf{f}))>0$ and $W$ is a positive dimensional irreducible component of $V(\mathbf{f})$, by Lemma \ref{lemma:IW},
$n-\#I_W- \#J_{I_W}>0$.
\end{proof}

Now we will characterize the sets $I\subset\{1,\dots, n\}$ such that
the irreducible components of $V(\mathbf{f}_I)$ intersecting
$(\C^*)^{n-\# I}$ yield irreducible components of $V(\mathbf{f})$.

\begin{prop}\label{prop:components}
Under the previous assumptions, let $I \subset \{1,\dots,n\}$. Then
$V(\mathbf{f}_I ) \cap (\C^*)^{n-\#I}$ is not empty if and only if
for every $J \subset J_I$, $\dim ( \sum_{j \in J} \cA_j^I) \ge \#J$
and, in this case, $V(\mathbf{f}_I) \cap (\C^*)^{n-\#I}$ is an
equidimensional variety of dimension $n - \#I-\#J_I$. In addition,
if  $W$ is an irreducible component of $V(\mathbf{f}_I ) \cap
(\C^*)^{n-\#I}$, we have that $\varphi_I(W)$ is an irreducible
component of $V(\mathbf{f}) \cap \bigcap_{i\notin I} \{x_i \ne 0\}$
if and only if for every $\widetilde{I}\subset I$, $\#\widetilde{I}
+ \#J_{\widetilde{I}} \ge \#{I} + \#J_{I}$.
\end{prop}

\begin{proof}
The first statement of the Proposition follows from Lemma
\ref{genericdimension}.

Without loss of generality, assume that $I=\{r+1,\dots, n\}$. Let
$W$ be an irreducible component of $V(\mathbf{f}_I) \cap (\C^*)^r$.

Suppose that for a subset $\widetilde I\subset I$ the inequality
$\#\widetilde I + \# J_{\widetilde I} < \# I + \# J_I$ holds. Assume
$\widetilde I=\{\tilde r +1,\dots, n\}$ for $\tilde r>r$.

Note that if $\xi= (\xi_1,\dots, \xi_r)\in V(\mathbf{f}_I)$, then
$(\xi, \mathbf{0}_{n-r})\in V(\mathbf{f})$ and, therefore,  $(\xi,
\mathbf{0}_{\tilde r-r})\in V(\mathbf{f}_{\widetilde I})$. Thus,  we
may consider $W\times \{ \mathbf{0}_{\tilde r - r}\} \subset
V(\mathbf{f}_{\widetilde I})$, which will be included in an
irreducible component $\widetilde W $ of $ V(\mathbf{f}_{\widetilde
I})$. Taking into account that $\mathbf{f}_{\widetilde I}$ consists
of $\#J_{\widetilde I}$ polynomials in $n-\#\widetilde I$ variables
and applying Lemma \ref{genericdimension} to $W$ and $\mathbf{f}_I$,
it follows that
$$\dim(\widetilde W) \ge n-\#\widetilde I-\#J_{\widetilde I} > n-\#I - \#J_{I} = \dim(W).$$
We conclude that $\varphi_I(W) = W\times \{ \mathbf{0}_{n-r}\}
\subsetneq \widetilde W\times \{ \mathbf{0}_{n-\tilde r}\}  \subset
V(\mathbf{f})$ and, therefore, $\varphi_I(W)$ is not an irreducible
component of $V(\mathbf{f})\cap \bigcap_{i\notin I} \{x_i \ne 0 \}$.

Conversely, if $\varphi_I(W) =W\times \{ \mathbf{0}_{n-r}\}$ is not
an irreducible component of $V(\mathbf{f})$, there is an irreducible
component $\widetilde W$ of this variety such that $W\times \{
\mathbf{0}_{n-r}\}\subsetneq \widetilde W$. The previous inclusion
implies that $I_{\widetilde W}\subset I$. Assume $ I_{\widetilde W}
= \{ \tilde r +1,\dots, n\}$ for $\tilde r > r$. Due to Lemma
\ref{lemma:IW}, $\pi_{I_{\widetilde W}} (\widetilde W)$ is an
irreducible component of $V(\mathbf{f}_{I_{\widetilde W}})$ having a
non-empty intersection with $(\C^*)^{\tilde r}$. Therefore,
$$ n-\# I_{\widetilde W} - \# J_{I_{\widetilde W}} = \dim(\pi_{I_{\widetilde W}} (\widetilde W))=\dim(\widetilde W)  > \dim W = n-\#I - \#J_I,$$
and so, $\#I + \#J_I > \# I_{\widetilde W} + \# J_{I_{\widetilde
W}}$.
\end{proof}

As a consequence of Proposition \ref{prop:components}, we have that
the irreducible components of $V(\mathbf{f})\subset \C^n$ are
contained in the linear subspaces $\bigcap_{i\in I}\{x_i =0\}$
associated to the subsets $I\subset \{1,\dots, n\}$ in
$$ \Gamma=\Big\{ I
\subset \{1, \dots, n\} \mid
 \forall J \subset J_I, \ \dim(\sum\limits_{j \in
J}\cA^I_j)\geq\#J; \ \forall \widetilde I\subset I,\ \#J_{\widetilde
I}+\#\widetilde I\geq \#J_I+\#I\Big\}.$$

Note that there may be sets $I_1\subsetneq I_2 \subset \{1,\dots,
n\}$ both in $\Gamma$, as it can be seen in Example \ref{ejemplo1},
where the three sets $\{1 \} $, $\{ 2 \}$ and $\{1,2\}$ give
irreducible components of the variety.

\bigskip

If we write $V^*(\mathbf{f}_I)$ to denote the union of all the
irreducible components of $V(\mathbf{f}_I)$ having a non-empty
intersection with $(\C^*)^{n-\# I}$, from Lemma \ref{lemma:IW} and
Proposition \ref{prop:components}, we deduce that, for every $I\in
\Gamma$,
$$\varphi_I(V^*(\mathbf{f}_I))=\bigcup_{W \textrm{ irred.~comp.~of }V(\mathbf{f}) \atop \textrm{such that } I_W =I} W .$$
We obtain the following characterization of the equidimensional
decomposition of $V(\mathbf{f})$ and, using Lemma
\ref{genericdimension}, a combinatorial expression for the degree of
$V(\mathbf{f})$:

\begin{thm}
\label{teoremacomb} Let $\mathbf{f} = (f_1, \dots, f_s)$ be generic
polynomials in $n$ variables supported on $\cA= (\cA_1,\dots,
\cA_s)\subset (\Z_{\ge 0}^n)^s$. For $k=0,\dots, n$, let
$V_k(\mathbf{f})$ be the equidimensional component of dimension $k$
of $V(\mathbf{f})$. Then, using the previous notations:
$$ V_k(\mathbf{f}) = \bigcup_{I \in\Gamma, \atop \# I + \# J_I = n-k} \varphi_I(V^*(\mathbf{f}_I)) .$$
Moreover, $\deg( V(\mathbf{f})) = \sum_{I\in \Gamma} MV_{n-\#I}(\cA^I,
\Delta^{(n-\# I - \# J_I)})$.
\end{thm}

\medskip
\begin{exmp}\label{ex:genericdecomp} {\rm Consider the following system of generic polynomials in $\Q[X_1,
X_2, X_3, X_4]$:
$$\begin{cases} a_1 X_1X_4 + a_2 X_1^2 X_4^2+ a_3 X_1X_2X_3 +a_4 X_2X_3 = 0 \\ b_1 X_1X_2 + b_2
X_1 X_2^2+ b_3 X_1X_3X_4 +b_4 X_3X_4 + b_5 X_3 X_4^2= 0 \\  c_1 X_1
X_2 X_4 + c_2 X_1 X_3 X_4 + c_3 X_2 X_3 + c_4 X_2 X_3 X_4 = 0 \\ d_1
X_1 + d_2 X_1^2 + d_3 X_1 X_2 + d_4 X_3^2 + d_5 X_3 X_4 = 0
\end{cases}
$$
Then, with the previous notation,  $\Gamma= \{\emptyset, \{1,2\},
\{1,3\}, \{2,3\}, \{2,4\},\{3,4\} \}$ and then,
\begin{itemize}
\item
$V_0(\mathbf{f}) = V^*(\mathbf{f}) \cup \{(0,0,\frac{b_4
d_5}{b_5d_4}, -\frac{b_4}{b_5}), (-\frac{d_1}{d_2}, 0, 0, \frac{a_1
d_2}{a_2 d_1}),(-\frac{d_1}{d_2}+
\frac{b_1d_3}{b_2d_2},-\frac{b_1}{b_2} , 0, 0)\}$:
\begin{itemize}
\item For $ I = \emptyset$ we have $19$ isolated solutions in $(\C^*)^4$ (this quantity is given by the mixed volume of
the family of supports associated to the system)
\item For $I = \{1,2\}$ we have:
$$\mathbf{f}_{\{1,2\}}=\begin{cases}
b_4 X_3 X_4 + b_5 X_3 X_4^2\\
d_4 X_3^2 +d_5 X_3 X_4
\end{cases}; \quad V^*(\mathbf{f}_{\{1,2\}}) = \Big\{\Big(\frac{b_4 d_5}{b_5d_4},
-\frac{b_4}{b_5}\Big)\Big\},$$ which gives the point $(0,0,\frac{b_4
d_5}{b_5d_4}, -\frac{b_4}{b_5})$.
\item For $I = \{2,3\}$ we have:
$$\mathbf{f}_{\{2,3\}}=\begin{cases}
a_1 X_1 X_4 + a_2 X_1^2 X_4^2\\
d_1 X_1 +d_2 X_1^2
\end{cases}; \quad V^*(\mathbf{f}_{\{2,3\}}) = \Big\{\Big(-\frac{d_1}{d_2},\frac{a_1 d_2}{a_2d_1}
\Big)\Big\},$$ which gives the point $(-\frac{d_1}{d_2}, 0, 0,
\frac{a_1 d_2}{a_2 d_1})$.
\item For $I = \{3,4\}$ we have:
$$\mathbf{f}_{\{3,4\}}=\begin{cases}
b_1 X_1 X_2 + b_2 X_1 X_2^2\\
d_1 X_1 +d_2 X_1^2+ d_3 X_1 X_2
\end{cases}; \quad V^*(\mathbf{f}_{\{3,4\}}) = \Big\{\Big(-\frac{d_1}{d_2}+\frac{b_1d_3}{b_2d_2},-\frac{b_1}{b_2}
\Big)\Big\},$$ which gives the point $(-\frac{d_1}{d_2}+
\frac{b_1d_3}{b_2d_2},-\frac{b_1}{b_2} , 0, 0)$.
\end{itemize}
\item $V_1(\mathbf{f}) = \{ x\in \C^4 \mid x_2=0, x_4=0, d_1x_1 +
d_2x_1^2+d_4x_3^2=0\}$:
\begin{itemize}
\item For $I = \{2,4\}$ we have:
$$\mathbf{f}_{\{2,4\}}=\{
d_1 X_1 +d_2 X_1^2+ d_4 X_3^2; \quad V^*(\mathbf{f}_{\{2,4\}}) =
\{(x_1,x_3) \mid d_1 x_1 +d_2 x_1^2+ d_4 x_3^2=0\},$$ which gives
 the curve $\{ x_2=0, x_4=0, d_1 x_1 +d_2 x_1^2+ d_4 x_3^2=0\}$.
\end{itemize}
\item $V_2(\mathbf{f}) = \{ x\in \C^4 \mid x_1=0, x_3=0\}$:
\begin{itemize}
\item For $I = \{1,3\}$ we have:
$$ \mathbf{f}_{\{1,3\}} = \emptyset; \quad V^*(\mathbf{f}_{\{1,3\}})
= \C^2,$$ which gives the plane 
$\{ x_1=0, x_3=0\}$.
\end{itemize}
\end{itemize}}
\end{exmp}

\medskip

\begin{rem}{\rm
In the case of a generic system $\mathbf{f}= (f_1,\dots, f_s)$ in
$n$ variables such that $\cA_1=\dots = \cA_s$, we have that the sets
$I \in \Gamma$, $I\ne \emptyset$,  are all $I \subset \{1,\dots, n\}
$ such that $\# J_I = 0$ and for every $\widetilde I \subsetneq I$,
$\# J_{\widetilde I} >0$; and $\emptyset \in \Gamma$ if and only if
$\dim(\cA_1) \ge s$.

Moreover, for every $I\in \Gamma$, $I\ne \emptyset$,
$\varphi_I(V^*(\mathbf{f}_I)) = \{ x_i = 0 \ \forall\, i \in I\}$
and so, apart from the components intersecting $(\C^*)^n$ that
correspond to $I= \emptyset$ (if $\emptyset \in \Gamma$), the only
irreducible components of $V(\mathbf{f})$ are linear subspaces of
$\C^n$.}
\end{rem}

\subsubsection{Algorithmic results}

According to Proposition \ref{prop:components}, the subsets  $ I
\subset \{1, \dots, n\}$ which
 yield irreducible components of the variety $V(\mathbf{f})$ are the ones satisfying simultaneously
\begin{enumerate}
\item \label{1} $\forall J \subset J_I $,  $\dim (\sum\limits_{j \in
J}\cA^I_j)\geq\#J$,
\item \label{2} $\forall I'\subset I$, $\#J_{I'}+\#I'\geq
\#J_I+\#I.$
\end{enumerate}

\bigskip
Now we present an algorithm to obtain the sets
$I$ satisfying condition (\ref{2}) and the inequality $\#I + \# J_I\le
n$, which is a  necessary condition for the system $\mathbf{f}_I$ to
have zeroes in $(\C^*)^{n-\#I}$, weaker but easier to check than
condition (\ref{1}). Our algorithm to find all the affine components
of $V(\mathbf{f})$ (see Algorithm \texttt{GenericAffineComps} below)
checks only among these sets whether condition (\ref{1}) is fulfilled
or not by means of a mixed volume computation.

\bigskip

\noindent \hrulefill

\medskip

\noindent \textbf{Algorithm} \texttt{SpecialSets}

\medskip

\noindent INPUT: A family of supports $\cA=(\cA_1,\dots,
\cA_s)\subset (\Z_{\ge 0}^n)^s$.

\smallskip
\begin{enumerate}

\item $P_{\emptyset}:= \min\{n,s\}$.
\item If $P_{\emptyset}= s$, add $(\emptyset, \{1,\dots,s\})$ to an empty list $\widetilde \Gamma$.

\item For $k=1,\dots, n$:

For every $I$ such that $\# I = k$:
\begin{enumerate}

\item \label{alg:maximum} $P_I := \min \{ n, \{P_{ I'} \}_{I' \subset I,\, \# I' = k - 1},  k+  \#J_I\}.$

\item \label{alg:sirve} If $P_I = k + \# J_I$, add $(I,J_I)$ to the list $\widetilde \Gamma$.
\end{enumerate}
\end{enumerate}

\smallskip
\noindent OUTPUT: The list $\widetilde \Gamma$ of all pairs of
subsets $(I, J_I)$, with  $I\subset \{1,\dots, n\}$ such that $\#I + \#
J_I\le n$ and for every $ \widetilde I\subset I$, $ \# \widetilde I
+\# J_{\widetilde I} \ge \#I + \# J_I$.

\medskip
\noindent \hrulefill

\bigskip

First, let us prove the correctness of this algorithm:

\begin{lem}
For every $I\subset \{1,\dots, n\}$, $P_I =\min\limits_{\widetilde I
\subset I} \{n,   \# \widetilde I + \# J_{\widetilde I}\}$.
\end{lem}

\begin{proof} By induction on $\# I$.

For $\# I =0$, since $\# J_\emptyset=s$, we have that $P_\emptyset=
\min\{n,  \#\emptyset\ + \# J_\emptyset \}$.

Assuming the identity holds for every subset of cardinality $k-1$,
let $I\subset \{1,\dots, n\}$ with $\# I = k$. Consider a proper
subset $\widetilde I_0\subsetneq I$. Then, there exists $I'\subset
I$ with $\#I' = k-1$ such that $\widetilde I_0\subset I'$. By the
inductive assumption, $$P_{I'}= \min\limits_{ \widetilde I \subset
I'} \{ n,  \# \widetilde I+ \# J_{\widetilde I}  \}$$ and so,
$P_{I'}\le \#\widetilde I_0+ \# J_{\widetilde  I_0} $. On the other
hand, by the definition of $P_I$, we have that $P_I\le P_{I'} $.
Thus, $P_{I}\le \#\widetilde  I_0+ \# J_{\widetilde I_0}  $.
\end{proof}

Now we estimate the complexity of the algorithm. Let $N =
\sum_{j=1}^s  \# \cA_j $. For each $I$ of cardinality $k\ge 1$, it
takes $k N + \# J_I +1 $ operations to compute $k + \# J_I$. Taking
the minimum among $k+2$ numbers takes $k+1$ comparisons. Thus, the
complexity of Step \ref{alg:maximum} is  $k (N+1) +\#J_I +2$. In
Step \ref{alg:sirve}, we add one comparison. As we have to do this
for each subset of $\{1, \dots, n\}$ of cardinality $k\ge 1$ and for
the empty set,
 the complexity is bounded by $2+\sum_{k=1}^n  {n \choose k} (k (N+1)+ s +
3)= 2+n 2^{n-1}(N +1)+(2^n -1) (s+3)$ which is of order $O(nN
2^{n})$.

\bigskip

Unfortunately, the exponential complexity of the algorithm cannot be
avoided, as the following examples show:

\begin{exmp}\label{example7}  {\rm Let   $L_1,\dots, L_n \in \Q [X_1, \dots, X_n]$
be generic affine linear forms. Consider the set of generic
polynomials $ f_1 = X_1.L_1, \dots, f_n= X_n.L_n$. In this case, if
$\cA = ( \cA_1, \dots, \cA_n)$ is their family of supports, $\Gamma
= \{I \mid I \subseteq \{1, \dots, n\}\}$ and therefore $\# \Gamma =
2^n$.}
\end{exmp}

One may think the exponentiality of the cardinal of the set $\Gamma$
in this example arises from the fact that the variables are factors
of the polynomials. However, this is not always the case as the
following example shows. This example also shows how subroutine
\texttt{SpecialSets} is useful to discard subsets which do not lead
to affine components: in this case, the only element in
$\widetilde{\Gamma}$ which does not correspond to a set in $\Gamma$
is $(\emptyset, \{1, \dots, 2n\})$.

\begin{exmp}{\rm
Consider generic polynomials $f_1, \dots, f_{2n} \in \Q[X_1,\dots,
X_{2n}]$,
$$f_j= \sum_{1\le k \le n} c_{jk}\, X_{2k-1} X_{2k},\quad j=1,\dots, 2n, $$
 all supported on the
set $\cA = \{e_{2k-1}+e_{2k}  \mid  1 \le k \le n\}$ where $e_i$
denotes the $i$th vector of the canonical basis of $\R^{2n}$.  Then,
it is easy to see that, for any subset $S \subset \{1,\dots, n\}$,
the set $I_S = \{2k-1\mid k \in S \} \cup \{2k \mid k \in \{1,\dots,
n\} \setminus S \}$ is in $\Gamma$, and that the sets $I_S$  are the
only ones (together with the empty set) obtained as first coordinate
of elements of $\widetilde{\Gamma}$ by the previous algorithm.
Therefore, in this case, we have that the number of elements of the
list $\widetilde \Gamma$ is $2^n +1$.}
\end{exmp}

In the following example, the usefulness of subroutine
\texttt{SpecialSets} is more evident:

\begin{exmp}{\rm  Let $\cA= (\cA_1, \dots, \cA_n )$ be a family of finite sets of $(\Z_{\ge 0})^n$
such that, for every $1 \le j\le n$, there exists a nonnegative
integer $d_{j_i}$ such that $d_{j_i}. e_i \in \cA_j$ for every $1
\le i \le n$. Then the output of subroutine \texttt{SpecialSets} in
this case is $\widetilde{\Gamma} = \{ (\emptyset, \{1,\dots,n\})\}$.}
\end{exmp}

\bigskip

The following algorithm computes a family of geometric resolutions
describing the affine variety defined by a generic system
$\mathbf{f}$ with given supports $\cA$.

\bigskip

\noindent \hrulefill

\medskip

\noindent \textbf{Algorithm} \texttt{GenericAffineSolve}

\medskip

\noindent  INPUT: A sparse representation of the generic system
$\mathbf{f} =(f_1,\dots, f_s)$ of polynomials in $\Q[X_1, \dots,
X_n]$.

\smallskip
\begin{enumerate}
\item\label{alg:findI} Apply algorithm \texttt{SpecialSets} to the family of the supports $\cA=(\cA_1,\dots,
\cA_s)$ of $\mathbf{f}$ to obtain the list $\widetilde{\Gamma}$ of
pairs of sets $(I, J_I)$ with $I \subset\{1,\dots, n\}$ such that
$\#I+\#J_I\le n$ and $\forall \widetilde I\subset I,\ \#\widetilde
I+\#J_{\widetilde I}\geq \#I+\#J_I$.
\item For every $(I, J_I)\in \widetilde{\Gamma} $:
\begin{enumerate}
\item For $j\in J_I $, compute $\cA_j^I := \{\pi_I(a) \mid a\in \cA_j \textrm{ such that } a_i= 0 \
\forall i\in I\}$, and obtain the sparse representation of the
system $\mathbf{f}_I$ supported on $\cA^I = (\cA_j^I)_{j \in J_I}$.
\item \label{alg:applyGTS} Apply algorithm \texttt{GenericToricSolve} to the sparse system
$\mathbf{f}_I$ to obtain a geometric resolution $R^I $ (possibly
empty) of the affine components of the set of solutions of
$\mathbf{f}_I$ intersecting the torus $(\C ^* )^{n - \# I}$.
\item If $R^{I} \ne \emptyset$:
\begin{enumerate}
\item Obtain the geometric resolution $\varphi_I(R^I)$ of the union of all irreducible components $W$ of $V(\mathbf{f})$ such that $I_W = I$ by adding zeroes to $R^I$ in the coordinates indexed by $I$.
\item If $n-(\#I + \# J_I) =k$, add $\varphi_I(R^I)$ to the list $\mathcal{V}_k$.
\end{enumerate}
\end{enumerate}
\end{enumerate}

\smallskip
\noindent OUTPUT: A family of lists $\mathcal{V}_k$, $0\le k \le
n-1$, of geometric resolutions, each list either empty or describing
the equidimensional component of $V(\mathbf{f})$ of dimension $k$.

\medskip

\noindent \hrulefill

\bigskip

The correctness of this algorithm is straightforward from
Proposition \ref{prop:components} and Theorem \ref{teoremacomb}.

When applying algorithm \texttt{GenericToricSolve} in Step
\ref{alg:applyGTS}, we need a
pre-processing obtaining a fine mixed subdivision. To do so, we may apply the
dynamic enumeration procedure from \cite{MTK07}. This procedure
proved to be very efficient even for large systems, but there are no explicit complexity bounds; for this reason,
we do not include its cost in our complexity estimates.
This pre-processing, in
particular, decides whether a set $I$ satisfies $\dim(\sum_{j\in J}
\cA_j^I)\ge \# J$ for every $J \subset J_I$ and, therefore, it
discards the sets $I \in \widetilde{\Gamma} \setminus \Gamma$. For
this reason, we will only consider the complexity of the
computations for the sets $I\in \Gamma$. This complexity  can be
estimated from the complexities of the subroutines applied at the
intermediate steps (see Proposition \ref{prop:toric}) and  is of
order
$$ O\Big( \sum _{I \in \Gamma} (n-\#I)^3 \big( N_I  + (n - \#I-\# J_I) (n - \#
I)\big) \log (d_I) M(D_I) \left(M(\mathfrak{M}_I) \left(M(D_I) + M
(E_I)\right) + D_I^2\right)\Big)
$$
where, for every $I \in \Gamma$, $N_I$, $d_I$, $D_I$,
$\mathfrak{M}_I$ and $E_I$ are the parameters defined in the
complexity of Algorithm \texttt{GenericToricSolve}, associated to
the system $\mathbf{f}_I$.

In order to estimate the overall complexity of the algorithm, note
that $N_I \le N:= \sum_{j=1}^s \# \cA_j$ and $d_I \le d:= \max_{1
\le j \le s} \{\deg (f_j)\}$ for every $I \in \Gamma$. In addition,
$\mathcal{D}:= \sum_{I \in \Gamma} D_I = \deg V(\mathbf{f})$. Note
that, if $\omega_{\max} $ is the maximum value of the lifting
functions applied to the supports $\cA^I$, for every $I \in \Gamma$
we have
$$ E_I \le MV_{n-\# I+1}(\Delta \times \{0\},(\cA_j^I
\times \{0, \omega_{\max}\})_{j \in J_I}, (\Delta \times \{0,
\omega_{\max}\})^ {(n -\# I - \# J_I)})   \le$$ $$ \le \omega_{\max}
\Big((n -\# I -\#J_I) MV_{n-\# I} (\cA^I, \Delta ^{(n -\# I
-\#J_I)}) + \sum_{\ell \in J_I} MV_{n - \#I}((\cA^I_j)_{j \ne \ell},
\Delta ^{(n -\# I -\#J_I+1)})\Big).$$

 Then,
if  $\mathcal{E}_{\max}:= \max _{I \in \Gamma} \{(n -\# I -\#J_I)
MV_{n-\# I} (\cA^I, \Delta ^{(n -\# I -\#J_I)}) +{}$ \linebreak
$\sum_{\ell \in J_I} MV_{n - \#I}((\cA^I_j)_{j \ne \ell}, \Delta
^{(n -\# I -\#J_I+1)})\}$ and $\mathfrak{M}_{\max} := \max _{I \in
\Gamma}\{\mathfrak{M} _{I} \}$, taking into account the complexity
of Step \ref{alg:findI}, we have:

\begin{thm}\label{thm:gendecomp}
Let $\mathbf{f}=(f_1,\dots, f_s)$ be a system of  generic
polynomials in $\Q[X_1,\dots, X_n]$ supported on $\cA=(\cA_1,\dots,
\cA_s)$.  \texttt{GenericAffineSolve} is a probabilistic algorithm
that
 computes
a family of lists $\mathcal{V}_k$, $0\le k \le n-1$, of geometric
resolutions, each list either empty or describing the
equidimensional component of $V(\mathbf{f})$ of dimension $k$. Using
the previous notation, the complexity of this algorithm is of order
$$O\left(n 2^{n}N  + n^3 (N + n^2)\log (d) M (\mathcal{D}) \left(
M(\mathfrak{M}_{\max}) \left( M (\mathcal{D}) + M(\omega_{\max}
\mathcal{E}_{\max})\right)  +\mathcal{D}^2 \right)\right).$$
\end{thm}

\section{Arbitrary sparse systems}\label{sec:arbitrary}

\subsection{An upper bound for the degree}\label{sec:degreebound}

The aim of this section is to show a bound for the degree of the
affine variety defined by a square system of sparse polynomials
which takes into account its sparsity.

The following example shows that the degree of an affine variety
defined by a generic sparse system with given supports is not an
upper bound for the degree of the variety defined by a particular
system with the same supports. One may think the problem arises from
the presence of irreducible components not intersecting the torus
either for the generic or the particular systems. However, this is
not the case:

\begin{exmp}
\label{ejemplo2} {\rm Consider the following  system:
$$\begin{cases}  X_1X_2-X_1-X_2+1=(X_1-1)(X_2-1) = 0 \\
X_1X_3-X_1-X_3+1=(X_1-1)(X_3-1) = 0 \\
X_2X_3-X_2-X_3+1=(X_2-1)(X_3-1) = 0
\end{cases}$$
The variety defined by the system consists of  $3$ lines, each
having a non-empty intersection with $(\C^*)^3$. However, if $\cA =
(\cA_1, \cA_2, \cA_3)$ is the family of the supports of the
polynomials, the degree of the variety defined by a generic system
with the same supports is $MV_3(\cA_1, \cA_2, \cA_3) = 2$.}
\end{exmp}

Our bound for the degree, which is stated in the following theorem,
is the mixed volume of a family of sets obtained by enlarging the
supports of the polynomials involved.

\begin{thm}
\label{bounddegree}
 Let $\mathbf{f} = (f_1, \dots, f_n)$ be $n$
polynomials in $\C[X_1,\dots, X_n]$ supported on $\cA= (\cA_1,\dots,
\cA_n)$ and let $V(\mathbf{f}) = \{ x \in \C^n \mid f_i( x) = 0
\hbox{ \ for every  } 1 \le i \le n \}.$ Let $\Delta = \{ 0, e_1,
\dots, e_n\}$ where $e_i$ the $i$th vector of the canonical basis of
$\R^n$. Then
$$ \deg (V(\mathbf{f})) \le MV_n(\cA_1 \cup \Delta,\dots, \cA_n\cup \Delta).$$
\end{thm}

Before proving the theorem, we will fix some notation and
definitions.

Let $r_j = \# (\cA_j \cup \Delta) - 1$ for $1 \le j \le n$. Consider
the morphism of varieties $\varphi: \C^n \to \mathbb{P}^n \times
\mathbb{P}^{r_1}\times \dots \times \mathbb{P}^{r_n}$ defined by
\begin{equation}\label{eq:varphi}
\varphi(x) = \left( (1:x), (x^a) _{a \in \cA_1 \cup \Delta}, \dots,
(x^a) _{a \in \cA_n \cup \Delta}\right).
\end{equation}
Let ${\mathcal X} =
\overline{\varphi(\C ^n)}$. For $1 \le j \le n$, we denote  by
$L_{j}$ the linear form in $\mathbb{P}^{r_{j}}$ given by the
coefficients of the polynomial $f_{j}$, that is to say, if $f_{j}=
\sum_{a \in \cA _j} c_{j,a} X^a$, then $L_{j} = \sum_{a \in \cA _j}
c_{j,a} X_{j,a}$.

For each integer $k$ ($0 \le k \le n$) and each subset $S \subset \{
1,\dots, k\}$  we define the variety ${\mathcal X}_{k, S}$
recursively in the following way:
 \begin{enumerate}
\item ${\mathcal X}_{0, \emptyset} = {\mathcal X}$.

\item  Provided ${\mathcal X}_{k, S}$ is defined for every $S \subset
\{ 1,\dots, k\}$, we define $\mathcal{X}_{k+1, T}$ with $T \subset \{ 1,\dots,
k+1\}$ as follows:
 \begin{itemize}
 \item If $k+1 \notin T$, $\mathcal{X}_{k+1, T}$ is the union of the irreducible components of $\mathcal{X}_{k, T}$
 included in $\{L_{k+1}=0\}$.
\item If $k+1 \in T$, $\mathcal{X}_{k+1, T}$ is  the intersection of $\{L_{k+1}=0\}$ with the union of the irreducible components of $\mathcal{X}_{k, T\setminus \{ k+1\}}$
not included in $\{L_{k+1}=0\}$.
 \end{itemize}
\end{enumerate}

Note that, from this definition, each ${\mathcal X}_{k, S}$ is an
equidimensional variety of dimension $n - \#S$. Moreover, if $\pi :
\mathbb{P}^n \times \mathbb{P}^{r_1}\times \dots \times
\mathbb{P}^{r_n} \to \mathbb{P}^n$ is the projection onto the first
factor, it is easy to see inductively that,  for every $1 \le k \le
n$,
$$   \bigcup _{S \subset \{1,\dots, k\}} \pi ({\mathcal X}_{k,
S}) = \overline{V(f_1, \dots, f_k)} \subset \mathbb{P}^n.$$

For an equidimensional subvariety $W$ of $\mathcal {X}$, we define
its multidegrees $\deg_{(r,0_k,1_{n-k})}(W)$, for $r,k\in \Z_{\ge
0}$ such that $n-k+r = \dim(W)$, as
$$\deg_{(r,0_k,1_{n-k})} (W) = \max \left\{ \# \Big(W\cap \bigcap_{j=1}^r \{\ell_{0,j} = 0\} \cap
\bigcap_{j = k+1}^n \{\ell_{j} = 0\} \Big)\right\}$$ where  the
maximum is taken over all $(\ell_{0,1}, \dots, \ell_{0,r},
\ell_{k+1}, \dots, \ell_{n})$ such that each $\ell_{0,j}$ is a
linear form in $\mathbb{P}^n$ and each $\ell_j$ is a linear form in
$\mathbb{P}^{r_j}$ and the intersection is finite. Note that the $1$
subscript indicates how many projective spaces are cut by a linear
form and the $0$ subscript shows how many remain uncut. As in the
case of the standard degree of affine or projective varieties
(see \cite{Hei83}), the maximum is attained generically.

In particular, it is clear that $ \deg_{(0,0_0,1_n)} ({\mathcal X})
= MV_n(\cA_1 \cup \Delta,\dots, \cA_n\cup \Delta)$.

\begin{lem}
\label{induccion} Under the previous assumptions, definitions and
notations,
$$\deg_{(k-\#S,0_k,1_{n-k})} ( {\mathcal X}_{k, S} )\ge \deg_{(k+1-\#S,0_{k+1},1_{n-k-1})} ({\mathcal X}_{k+1,
S}) + \deg_{(k-\#S,0_{k+1},1_{n-k-1})} ( {\mathcal X}_{k+1, S\cup
\{k+1\}})$$
\end{lem}

\begin{proof} As the variety ${\mathcal X}_{k, S} = {\mathcal
X}_{k+1, S} \cup \widetilde{{\mathcal X}_{k, S}}$,  where
$\widetilde{{\mathcal X}_{k, S}}$ is the union of the irreducible
components of ${\mathcal{X}_{k, S}}$ not contained in $\{L_{k+1}= 0\}$, using
genericity in the definition of multidegrees, we have that
$$\deg_{(k-\#S,0_k,1_{n-k})} ({\mathcal X}_{k, S}) = \deg_{(k-\#S,0_{k},1_{n-k})} ({\mathcal X}_{k+1,
S}) + \deg_{(k-\#S,0_{k},1_{n-k})} (\widetilde{{\mathcal X}_{k,
S}}).$$

Note that, by adding the simplex $\Delta$ to the supports $\cA_1,\dots, \cA_n$, the points of the varieties
in $\mathbb{P}^n \times
\mathbb{P}^{r_1}\times \dots \times \mathbb{P}^{r_n}$ we
are considering have a copy of their coordinate in $\mathbb{P}^n$ in each coordinate in $\mathbb{P}^{r_j}$ for $j=1, \dots,n$ (see Equation (\ref{eq:varphi})).
Because of this, if $\ell_0$ is a linear form in
$\mathbb{P}^n$ and $\ell_{k+1}$ is exactly the same linear form
involving only the corresponding coordinates in $\mathbb{P}^{r_{k+1}}$, we
have that ${\mathcal X}_{k+1, S} \cap \{\ell_{0}= 0\} = {\mathcal
X}_{k+1, S} \cap \{\ell_{k+1}= 0\}$ and therefore,
$$\deg_{(k-\#S,0_{k},1_{n-k})} ({\mathcal X}_{k+1, S}) \ge
\deg_{(k+1-\#S,0_{k+1},1_{n-k-1})} ({\mathcal X}_{k+1, S})$$ and
$$\deg_{(k-\#S,0_{k},1_{n-k})}
(\widetilde{{\mathcal X}_{k, S}}) \ge
\deg_{(k-\#S,0_{k+1},1_{n-k-1})} \left(\widetilde{{\mathcal X}_{k,
S}} \cap \{L_{k+1} = 0\}\right) =$$  $$
=\deg_{(k-\#S,0_{k+1},1_{n-k-1})} ({\mathcal X}_{k+1, S\cup
\{k+1\}}).$$
\end{proof}
Now, we can prove Theorem \ref{bounddegree}:
\begin{proof}
Consider the projection $\pi : \mathbb{P}^n \times
\mathbb{P}^{r_1}\times \dots \times \mathbb{P}^{r_n} \to
\mathbb{P}^n$ onto the first factor. As
$$   \bigcup _{S \subset \{1,\dots, n\}} \pi ({\mathcal X}_{n,
S}) = \overline{V(\mathbf{f})} \subset \mathbb{P}^n,$$ we have that
$$\deg (V(\mathbf{f})) = \deg \Big( \bigcup _{S \subset \{1,\dots,
n\}} \pi ( {\mathcal X}_{n, S}) \Big) \le \sum_{S \subset \{1,\dots,
n\}}\deg \pi ( {\mathcal X}_{n, S}) = $$ $$= \sum_{S \subset
\{1,\dots, n\}}\deg_{(n-\# S, 0_n, 1_0)} ({\mathcal X}_{n, S}) $$
(this last equality is nothing but our definition of multidegree).

Applying inductively Lemma \ref{induccion}, we get that, for each $0
\le k \le n$,

$$\sum_{S \subset \{1, \dots, k \}} \deg_{(k-\#S, 0_k, 1_{n-k})} ({\mathcal X}_{k, S} ) \le
\deg_{(0,0_0,1_n)} {\mathcal X}_{0, \emptyset}$$ and therefore we
obtain that

$$\sum_{S \subset \{1,\dots, n\}}\deg_{(n-\# S, 0_n, 1_0)} ({\mathcal X}_{n,
S}) \le \deg_{(0,0_0,1_n)} ({\mathcal X}_{0, \emptyset}) =
MV_n(\cA_1 \cup \Delta,\dots, \cA_n\cup \Delta).$$
\end{proof}

In the following examples, we show that this bound may be attained:
\begin{exmp}{\rm
Let $S \subsetneq \{1, \dots, n\}$ and let $f_1,\dots, f_n \in
\Q[X_1,\dots, X_n]$ be polynomials of degree $d$ in the variables
$(X_i)_{i\in S}$ with no zero coefficients (that is to say, the
monomials appearing in $f_1,\dots, f_n$ are those of degree less or
equal to $d$ in the variables $(X_i)_{i\in S}$). If $\cA$ is their
common support, it is evident that $MV_n(\cA^{(n)})= 0$ and that
$SM_n(\cA^{(n)})= 0$. However, $MV_n( (\cA \cup \Delta)^{(n)}) =
d^{\# S}$ and this degree can be attained for special choices of the
polynomials: If $f_1, \dots, f_{\# S}\in \Q [(X_i)_{i \in  S}]$ is a
family with $d^{\# S}$ isolated solutions in $\C^{\# S}$, take
linear combinations $f_{\# S+1}, \dots, f_n $  of $f_1,
\dots, f_{\# S}$. Then, the family of polynomials $f_1, \dots, f_{\#
S},f_{\# S+1}, \dots, f_n \in \Q[X_1,\dots, X_n]$ defines a variety
formed by $d^{\# S}$ affine linear spaces of dimension ${n-{\# S}}$.}
\end{exmp}

Our bound can be also attained for generic systems:
\begin{exmp} {\rm Consider the family of generic polynomials $f_1, \dots, f_n$  defined in Example
\ref{example7} and let their supports  $\cA = (\cA_1, \dots,
\cA_n)$. Then, the affine variety they define consists of $2^n$
points, and therefore, its degree is $2^n = MV_n (\cA_1\cup \Delta,
\dots, \cA_n\cup \Delta)$.}
\end{exmp}

\subsection{An algorithm in the non-generic case}\label{sec:algnongeneric}

In the sequel we will describe an algorithm that, given an
\emph{arbitrary} square system of sparse polynomials, provides a finite set
of points in each irreducible component of the affine variety
the system defines. The complexity of this algorithm depends on the
bound for the degree of the variety obtained in the previous
section.

Let $\mathbf{f} = (f_1, \dots, f_n)$ be $n$ polynomials in
$\Q[X_1,\dots, X_n]$ supported on $\cA= (\cA_1,\dots, \cA_n)$ and,
for a fixed $k $, $1 \le k \le n-1$, let $L_1, \dots, L_k$ be
generic affine linear forms in  $\Q[X_1,\dots, X_n]$. Then,
 if $V_k(\mathbf{f})$ is the equidimensional component of dimension $k$ of $V(\mathbf{f})$,
 the isolated common zeroes  of $f_1, \dots, f_n, L_1, \dots, L_k$ are
 $\deg (V_k(\mathbf{f}))$ points in $V_k(\mathbf{f})$. The idea is
 to represent each equidimensional component by means of the corresponding set
 of points  (c.f. the notion of witness point set in \cite{SW05}).

\begin{exmp}{\rm
Consider the following polynomial system:
$$\mathbf{f}= \left\{ \begin{array}{l}
 X_1^3 X_2X_3-X_1 X_2 X_3^3-X_1^2+X_3^2=(X_1X_2X_3-1)
 (X_1-X_3)(X_1+X_3)\\[2mm]
X_1^2X_2^2X_3-X_1^2X_2X_3-X_1X_2^2X_3^2+X_1X_2X_3^2-X_1X_2+X_1+X_3X_2-X_3={}\\
\quad {}= (X_1X_2X_3-1) (X_1-X_3)(X_2-1)\\[2mm]
X_1X_2^3X_3-X_1X_2X_3^2-X_2^2+X_3 =(X_1X_2X_3-1)(X_2^2-X_3)
\end{array}\right.$$
The equidimensional components of $V(\mathbf{f})$ are
$$V_0(\mathbf{f})=\{ (-1,1,1)\},\quad
V_1(\mathbf{f})=\{ x_1-x_3 =0, x_2^2-x_3=0\}, \quad
V_2(\mathbf{f})=\{ x_1x_2x_3 -1=0\}.$$ Taking $L_1= X_1 - X_2$ and
$L_2 =6 X_2-X_3+7$, we have:
\begin{itemize}
\item  the set of isolated points of $V(\mathbf{f})\cap
\{x_1-x_2=0\}$ is $\{(1,1,1),(0,0,0) \}$, which is a set with
$2=\deg(V_1(\mathbf{f}))$ points in $V_1(\mathbf{f})$.
\item $V(\mathbf{f})\cap \{x_1-x_2=0,6 x_2-x_3+7=0\}=
\{(-1,-1,1),(-\frac{1}{2},-\frac{1}{2},4),(\frac{1}{3},\frac{1}{3},9)
\}$, which is a set with $3=\deg(V_2(\mathbf{f}))$ points in
$V_2(\mathbf{f})$.
\end{itemize}}
\end{exmp}

For a fixed $k$, $0\le k \le n-1$, in order to compute the isolated
common zeroes of $f_1, \dots, f_n, L_1, \dots, L_k$, by taking $n$
generic linear combinations of these polynomials,
 we obtain a system of $n$ polynomials in $n$ variables having these points among its
 isolated zeroes (see \cite{Hei83}). Note that, in order to achieve this, it suffices to take linear combinations of the form
$$ f_i(X) + \sum_{j=1}^k b_{ij}L_j(X), \qquad i=1,\dots, n$$
for generic $b_{ij}$ $(1\le i \le n, 1\le j \le k)$.

Procedure \texttt{PointsInEquidComps} described below computes a
family of $n$ geometric resolutions $R^{(k)}$, for $0\le k \le n-1$,
encoding a finite set of points and such that $R^{(k)}$ represents
at least $\deg (W)$ points in each irreducible component $W$ of
dimension $k$ of $V(\mathbf{f})$.

The intermediate subroutine \texttt{CleanGR} takes as input a
geometric resolution $(q(u),$ $v_1(u),\dots,
v_n(u))\subset(\Q[u])^{n+1}$ of a finite set of points
$\mathcal{P}\subset \C^n$ and a list
$\mathbf{f}=(f_1,\dots, f_n)$ of polynomials in $\Q[X_1,\dots, X_n]$, and computes a
geometric resolution $(Q(u), V_1(u),\dots, V_n(u))$ of $\mathcal{P}
\cap V(\mathbf{f})$:
\begin{eqnarray*}
Q(u) &=& \gcd(q(u), f_1(v_1(u), \dots, v_n(u)), \dots, f_n(v_1(u), \dots, v_n(u)))\\
 V_i(u) &=& v_i(u)\ \textrm{mod}\, Q(u), \quad i=1,\dots, n.
 \end{eqnarray*}
Let $\cA_j$ be the support of $f_j$, $d$  an upper bound for
$\deg(f_j)$, $j=1,\dots, n$, and $D=\deg q$. First, the subroutine
computes slp's of length $O(n D \log D)$ for the polynomials $v_i$,
$i=1,\dots, n$. The gcd $Q(u)$ is computed successively as follows:
 For $j=1,\dots, n$, the subroutine computes an slp of length $\mathcal{L}_j= O(n\log d \#\cA_j)$ for the polynomial $f_j$ and, by multipoint evaluation and interpolation, the dense representation of $F_j(u) = f_j(v_1(u), \dots, v_n(u))$ within complexity  $O(M(dD)(\mathcal{L}_j + nD\log D))$; then, it computes \linebreak
$Q_j(u):=\gcd(Q_{j-1}, F_j(u))$ within $O(M(dD))$ additional
operations. Finally, the polynomials $V_i(u)$  for $i=1,\dots, n$
are obtained within complexity $O(nM(D))$. The overall complexity of
the procedure is of order $O(M(dD)(n\log d \sum_{j=1}^n \# \cA_j +
n^2 D \log D))$.

\bigskip

\noindent\hrulefill

\medskip

\noindent \textbf{Algorithm} \texttt{PointsInEquidComps}

\medskip
\noindent INPUT: A sparse representation of a system $\mathbf{f}
=(f_1,\dots, f_n)$ of polynomials in $\Q[X_1,\dots, X_n]$ supported
on $\cA = (\cA_1, \dots, \cA_n)$, a lifting function $\omega
=(\omega_1, \dots, \omega_n)$ for $\cA_\Delta = (\cA_1\cup \Delta,
\dots, \cA_n \cup \Delta)$ and the mixed cells in the induced
subdivision of $\cA_\Delta$.

\smallskip
\begin{enumerate}
\item \label{alg2:generico} Choose randomly coefficients for a polynomial system $\mathbf{g}= (g_1, \dots, g_n)$  supported on $\cA_\Delta$.
\item \label{alg2:solvegeneric} Apply the algorithm in \cite[Section 5]{JMSW09} to $\mathbf{g}$ to obtain a geometric resolution $R_\mathbf{g}$ of its zeroes in $\C^n$.
\item \label{alg2:chooseL} Choose randomly $n-1$ affine linear forms $L_1, \dots, L_{n-1}$ in the variables $X =(X_1,\dots,X_n )$ and $n (n-1)$ integer numbers $(b_{ij})_{1\le i \le n, 1 \le j \le n-1}$.
\item \label{alg2:recursion} For $k = 0, \dots, n-1$:
\begin{enumerate}
\item Obtain the sparse representation of the polynomials $h_i^{(k)} (X) = f_i(X) + \sum_{j=1}^k b_{ij}L_j(X)$ for $1 \le i \le n$.
\item \label{alg2:recursionGR} Apply the algorithm in  \cite[Section 6]{JMSW09} to $\textbf{h}^{(k)}= (h_1^{(k)}, \dots, h_n^{(k)})$  to obtain from $R_{\mathbf{g}}$ a geometric
resolution  of a finite set $\mathcal{P}_k$ which contains the
isolated common zeroes of $\textbf{h}^{(k)}$ in $\C^{n}$.
\item \label{alg2:clean} Apply subroutine \texttt{CleanGR} to the previous geometric resolution and $\mathbf{f}$ to obtain a geometric resolution
$R^{(k)} $ of  $\mathcal{P}_k \cap V(\mathbf{f})$.
\end{enumerate}
\end{enumerate}

\smallskip
OUTPUT: The $n$ geometric resolutions $R^{(k)}$ for $0 \le k \le
n-1$.

\medskip

\noindent\hrulefill

\bigskip

In the sequel we will  estimate the complexity of this procedure.

Steps \ref{alg2:generico} and \ref{alg2:chooseL} are fulfilled by
taking a random choice of numbers. We will not consider the cost of
this random choice in the overall complexity (see Remark \ref{rem:prob}). The complexity of Step
\ref{alg2:solvegeneric} is
 $$ O( (n^3N_\Delta \log (d) +
n^{1+\Omega})M(D_\Delta)M(\mathfrak{M}_\Delta)(M(D_\Delta)+M(E_\Delta)))$$
where
\begin{itemize}\itemsep=0pt
\item $N_\Delta:=\sum_{j=1}^n\#(\cA_j \cup \Delta)$;
\item $d:= \max_{1 \le j \le n} \{\deg f_j\}$;
\item $D_\Delta:=MV_{n}(\cA_1 \cup \Delta, \dots, \cA_n \cup
\Delta)$;
\item $\mathfrak{M}_\Delta:=\max\{\|\mu\|\}$ where the maximum ranges over all primitive    normal vectors to the   mixed cells in
the fine mixed subdivision of $\cA_\Delta$ induced by $\omega$;
\item $E_\Delta:=MV_{n+1}(\Delta \times \{0\},(\cA_1 \cup \Delta) (\omega_1), \dots, (\cA_n \cup \Delta) (\omega_n))$
where $(\cA_j \cup \Delta) (\omega_j)$ for every  $ 1 \le j \le n$
 is the set $\cA_j \cup \Delta$   lifted by  $\omega$.
 \end{itemize}
In Step \ref{alg2:recursionGR}, we compute a finite set which
includes the affine isolated zeroes of the system
$\textbf{h}^{(k)}$.  By applying the result in \cite[Proposition
6.1]{JMSW09}, we have that the complexity of  this step is bounded
by
 $ O((n^2N_\Delta\log d +
n^{1+\Omega})M(D_\Delta)M(E'_\Delta))$ where
 $E'_\Delta:=MV_{n+1}(\{0\}\times \Delta, \{0,1\}\times (\cA_1 \cup \Delta),  \dots,\{0,1\}\times (\cA_n \cup \Delta))$.
Finally, the complexity of Step \ref{alg2:clean} is of order
$O(M(dD_\Delta)(n\log d \sum_{j=1}^n \# \cA_j + n^2 D_\Delta \log
D_\Delta))$.

Note that the parameters $E_\Delta$ and $E'_\Delta$ in the previous
complexities can be bounded as follows:
$$E'_\Delta = \sum_{j=1}^n MV_n(\Delta, \cA_1\cup \Delta, \dots, \widehat{\cA_{j}\cup \Delta}, \dots, \cA_n \cup \Delta)\le n D_\Delta$$
and, if $\omega_{\max}:= \max_{j,a}\{ \omega_j(a) \mid 1\le j \le
n,\ a\in \cA_j\cup \Delta\}$,
$$ E_\Delta \le MV_{n+1}(\Delta \times \{0\},(\cA_1 \cup \Delta)\times \{0, \omega_{\max}\}, \dots, (\cA_n \cup \Delta) \times \{0, \omega_{\max}\}) \le \omega_{\max} n D_\Delta.$$
Taking into account these bounds, we have:

\begin{thm}\label{thm:nongeneric}
\noindent Let $\mathbf{f} = (f_1, \dots, f_n)$ be $n$ polynomials in
$\Q[X_1,\dots, X_n]$ supported on $\cA= (\cA_1,\dots, \cA_n)$.
\texttt{PointsInEquidComps} is a probabilistic algorithm which
computes a family of $n$ geometric resolutions $(R^{(0)}, R^{(1)},
\dots, R^{(n-1)})$ such that, for every $0\le k \le n-1$, $R^{(k)}$
represents a finite set containing $\deg V_k(\mathbf{f})$ points in
the equidimensional component $V_k(\mathbf{f})$ of dimension $k$ of
$V(\mathbf{f})$. Using the previous notation, the complexity of the
algorithm is of order $$O(\omega_{\max} n^4 N_\Delta \log(d)
M(dD_\Delta) M(D_\Delta) M(\mathfrak{M}_\Delta) ).$$
\end{thm}

\begin{exmp}{\rm
Consider the polynomial system introduced in Example \ref{ejemplo2},
given by the polynomials
$$\mathbf{f}= \begin{cases}  f_1= X_1X_2-X_1-X_2+1 \\
f_2=X_1X_3-X_1-X_3+1 \\ f_3=X_2X_3-X_2-X_3+1
\end{cases}$$
supported on $\cA_1 = \{(1,1,0), (1,0,0), (0,1,0), (0,0,0)\}$,
$\cA_2 = \{(1,0,1), (1,0,0), (0,0,1), (0,0,0)\}$ and $\cA_3 =
\{(0,1,1), (0,1,0), (0,0,1), (0,0,0)\}$ respectively.

 Algorithm \texttt{PointsInEquidComps} first chooses (at random) a system supported on $(\cA_1\cup \Delta, \cA_2\cup \Delta, \cA_3\cup \Delta)$,
 for example:
$$\mathbf{g}= \begin{cases}   2X_1X_2-2X_1+X_2-X_3+1 \\
X_1X_3-X_1+2 X_2+2X_3+2 \\ X_2X_3+X_1-2 X_2+X_3-1
\end{cases}$$
and computes a geometric resolution $R_{\mathbf{g}}$ of its isolated
common roots in $\C^3$:
$$R_{\mathbf{g}} = \left\{\begin{array}{l}
u^5-\frac{9}{2}u^4-17u^3+80u^2-2u-\frac{155}{2}=0\\[2mm]
X_1=-\frac{9}{100}u^4+\frac{7}{40}u^3+\frac{351}{200}u^2-\frac{543}{200}u-\frac{137}{40}\\[2mm]
X_2=-\frac{1}{200}u^4-\frac{1}{80}u^3-\frac{1}{400}u^2+\frac{233}{400}u+\frac{7}{80}\\[2mm]
X_3=-\frac{1}{10}u^4+\frac{3}{20}u^3+\frac{7}{4}u^2-\frac{51}{20}u-\frac{13}{4}
\end{array}\right.$$
Then, in Step \ref{alg2:chooseL}, the algorithm takes $2$ linear
forms:
\begin{eqnarray*}
L_1&=& X_1 + X_2 + 2 X_3\\
L_2&=& X_1 + 2 X_2
\end{eqnarray*}
In Step \ref{alg2:recursion}, for $k=0, 1, 2$, the isolated roots of
the system $\mathbf{h}^{(k)}$ obtained by adding to $\mathbf{f}$
generic linear combinations of $L_i$, $i=0,\dots, k$, are computed:
$$\mathbf{h}^{(0)} = \mathbf{f}, \quad
\mathbf{h}^{(1)}= \left\{ \begin{array}{lcl}
f_1(X) + L_1(X) &=& X_1 X_2 + 2 X_3 +1\\
f_2(X) - L_1(X) &=& X_1X_3 - 2 X_1 - 3 X_3 - X_2+1\\
f_3(X) + 2 L_1(X)&=& X_2X_3+ X_2 + 3 X_3 + 2 X_1 +1
\end{array}\right.$$
$$\mathbf{h}^{(2)} = \left\{ \begin{array}{lcl}
f_1(X) + L_1(X) +L_2(X)&=& X_1 X_2 +X_1+2X_2 + 2 X_3+1\\
f_2(X) - L_1(X) +2L_2(X)&=& X_1X_3 +3 X_2 - 3 X_3 +1\\
f_3(X) + 2 L_1(X)+L_2(X)&=& X_2X_3+3 X_2 + 3 X_3 + 3 X_1 +1
\end{array}\right.$$
In order to do this, the algorithm deforms the geometric resolution
$R_{\mathbf{g}}$ to  geometric resolutions $R_{\mathbf{h}^{(k)}}$ of
the sets of isolated roots of $\mathbf{h}^{(k)}$:
$$R_{\mathbf{h}^{(0)}} = \left\{ \begin{array}{l}
u^3-7u^2+2u+40=0\\[1mm]
X_1= 1\\[1mm]
X_2=-\frac{1}{14}u^2+\frac{9}{14}u-\frac{3}{7}\\[1mm]
X_3 =\frac{1}{7}u^2 -\frac{2}{7}u-\frac{1}{7}
\end{array}
\right.
$$

$$R_{\mathbf{h}^{(1)}} = \left\{ \begin{array}{l}
u^5- \frac92 u^4 -13 u^3 +68 u^2-64 u=0\\[1mm]
X_1= \frac{57}{200}u^4-\frac{369}{400}u^3-\frac{953}{200}u^2+\frac{647}{50}u-3\\[2mm]
X_2=-\frac{269}{600}u^4+\frac{1573}{1200}u^3+\frac{1567}{200}u^2-\frac{2599}{150}u+1\\[2mm]
X_3 =
\frac{367}{600}u^4-\frac{2039}{1200}u^3-\frac{2181}{200}u^2+\frac{3407}{150}u+1
\end{array}
\right.
$$

$$R_{\mathbf{h}^{(2)}} = \left\{ \begin{array}{l}
u^5+\frac{49}{2}u^4-\frac{1549}{9}u^3+\frac{538}{9}u^2+\frac{679}{6}u-\frac{769}{18}=0\\[1mm]
X_1=
-\frac{101214}{1803049}u^4-\frac{2537721}{1803049}u^3+\frac{15987545}{1803049}u^2+\frac{3986650}{1803049}u
-\frac{7719426}{1803049}\\[2mm]
X_2=\frac{58338}{9015245}u^4+\frac{277533}{1803049}u^3-\frac{11347628}{9015245}u^2+\frac{5343077}{9015245}u
+\frac{8036523}{9015245}\\[2mm]
X_3 =
\frac{389394}{9015245}u^4+\frac{1982655}{1803049}u^3-\frac{57242469}{9015245}u^2-\frac{21604159}{9015245}u+\frac{22524084}{9015245}
\end{array}
\right.
$$
Finally, subroutine \texttt{CleanGR} removes spurious factors from
$R_{\mathbf{h}^{(k)}}$ to obtain a geometric resolution $R^{(k)}$ of
a finite set containing a set of representative points of the
equidimensional component of dimension $k$
$$R^{(0)} = \left\{ \begin{array}{l}
u^3-7u^2+2u+40=0\\[1mm]
X_1= 1\\[1mm]
X_2=-\frac{1}{14}u^2+\frac{9}{14}u-\frac{3}{7}\\[1mm]
X_3 =\frac{1}{7}u^2 -\frac{2}{7}u-\frac{1}{7}
\end{array}
\right.\quad
 R^{(1)}= \left\{ \begin{array}{l}
u^3+2u^2-8u=0\\[1mm]
X_1=\frac{1}{2}u^2+u-3 \\[1mm]
X_2=-\frac{1}{6}u^2+\frac{1}{3}u+1\\[1mm]
X_3 = -\frac{1}{6}u^2-\frac{2}{3}u+1
\end{array}
\right. \quad R^{(2)} = \emptyset
$$
Since $u^3-7u^2+2u+40= (u+2)(u-4)(u-5)$ and $u^3+2u^2-8u =
(u+4)u(u-2)$, substituting their roots in $R^{(0)}$ and $R^{(1)}$
respectively, we get the following sets of points in
$V(\mathbf{f})$:
$$\mathcal{W}_0=\{(1,-2,1), (1,1,1), (1,1,2)\}\qquad \mathcal{W}_1=\{(-3,1,1),(1,1,-1),(1,-3,1)\}.$$
Note that $\mathcal{W}_1$ contains exactly one representative point
for each of the $3$ lines in $V(\mathbf{f})$. Moreover, the fact
that $R^{(2)} = \emptyset$ implies that $V(\mathbf{f})$ does not
have irreducible components of dimension $2$. However,  although
there are no isolated points in $V(\mathbf{f})$, $\mathcal{W}_0
\subset V(\mathbf{f})$ is not empty.}
\end{exmp}

\begin{rem}\label{rem:prob}{\rm
All the random choices of points made by our algorithms lead to correct computations provided
these points do not annihilate certain polynomials whose degrees can be explicitly bounded.
These bounds depend polynomially on the degrees of affine varieties associated to the input
systems, which in turn, can be estimated in terms of mixed volumes due to Theorem \ref{bounddegree}.
The Scwhartz-Zippel lemma allows us to control the bit size of the constants to be chosen at random
in order that the error probability of the algorithms is less than a fixed number within
the stated complexity bounds.

Although we do not include the precise probability estimates here, for a similar analysis we refer
the reader to \cite{JMSW09}, where the genericity of zero-dimensional sparse systems is studied
and the probability of success of the algorithms to compute isolated solutions of sparse systems is stated in detail. We also refer the reader to \cite[Proposition 4.5]{KPS01} for bounds on the genericity of hyperplanes intersecting an equidimensional variety in as many points as its degree, and to \cite[Lemma 3 and Remark 4]{JS02} for the analysis of the genericity of linear combinations of
input equations required in Section \ref{sec:algnongeneric}.}
\end{rem}

\bigskip

\noindent \textbf{Acknowledgements.} The authors would like to thank the referees for their helpful comments.

\end{document}